\documentclass[preprint,times,3p]{elsarticle}

\usepackage{amsmath}
\usepackage{amssymb}
\usepackage{amsfonts}
\usepackage{graphicx}
\usepackage{epstopdf}
\usepackage{color}
\usepackage{bm}
\usepackage{multirow}
\usepackage{natbib}
\usepackage{hyperref}
\usepackage{cleveref}
\usepackage{amsthm}
\usepackage{lineno} 
\nolinenumbers

\usepackage{color}
\usepackage{float}
\usepackage{caption}

\usepackage[caption=false]{subfig}

\usepackage{ifpdf}
\ifpdf%
\usepackage{pdflscape}
\else
\usepackage{lscape}
\fi

\hypersetup{
	colorlinks,
	linkcolor={red!50!black},
	citecolor={blue!50!black},
	urlcolor={blue!80!black}
}

\newtheorem{remark}{Remark} 
\newtheorem{lemma}{Lemma}

\def\a{\mbox{\boldmath $a$}}

\def\e{\mbox{\boldmath $e$}}
\def\f{\mbox{\boldmath $f$}}

\def\g{\mbox{\boldmath $g$}}

\def\m{\mbox{\boldmath $m$}}
\def\n{\mbox{\boldmath $n$}}

\def\x{\mbox{\boldmath $x$}}

\def\0{\mbox{\boldmath $0$}}

\begin{document}

\begin{frontmatter}

\title{A New Fractional Step Structure Preserving Method for The Landau-Lifshitz-Gilbert Equation}

\author[XJTLU]{Changjian Xie\corref{cor}}
\cortext[cor]{Corresponding author.} 
\ead{Changjian.Xie@xjtlu.edu.cn}



\address[XJTLU]{School of Mathematics and Physics, Xi'an-Jiaotong-Liverpool University, Re'ai Rd. 111, Suzhou, 215123, Jiangsu, China.}

\begin{abstract}
In this paper, we propose a structure preserving method using a Crank-Nicolson's type method with an implicit Gauss-Seidel fractional iteration. Such a method is of first-order accuracy in time and second-order accuracy in space, stable and length preserving. Such kind of method brings great benefits for the theoretical analysis. The numerical accuracy, norm preserving and stability are verified for 1D and 3D tests. 

\end{abstract}


\end{frontmatter}

\section{Introduction}

The Landau-Lifshitz-Gilbert (LLG) equation \cite{Landau1935On,Gilbert:1955} incorporates two fundamental terms governing magnetization dynamics: a gyromagnetic term, responsible for energy conservation, and a damping term, which models energy dissipation. 
The LLG equation constitutes a vector-valued, nonlinear system characterized by a point-wise constant magnitude constraint on the magnetization vector. The model without damping is given by
\begin{align*}
    \m_t=-\m\times \Delta \m.
\end{align*}
Significant research efforts have been dedicated to devising efficient and numerically stable methods for micromagnetic simulations, as summarized in review articles such as \cite{kruzik2006recent,cimrak2007survey}. 
Such a equation can be solved with Crank-Nicolson's method (midpoint type) below,
\begin{align*}
    \frac{\m_h^{n+1}-\m_h^n}{\Delta t}=-\frac{\m_h^{n+1}+\m_h^n}{2}\times \Delta_h \left(\frac{\m_h^{n+1}+\m_h^n}{2}\right),
\end{align*}
which is a nonlinear scheme with great stability. However, this paper focus on the linearity, efficiency, stability and structure preservation. 
A notable limitation persists: the implicit Crank-Nicolson's method \cite{JEONG2010613} involves a nonlinear system with the norm preserving property. In this paper, we use such a method and such property to construct a structure preserving method based on implicit Gauss-Seidel iteration.

The rest of this paper is organized as follows. \cref{sec: numerical scheme} begins with a review of the micromagnetic model, followed by a detailed description of the proposed numerical scheme. \cref{sec:experiments} presents extensive numerical results, encompassing verification of temporal and spatial accuracy in one-dimensional (1D) and three-dimensional (3D) settings. 
Concluding remarks and potential future research directions are provided in \cref{sec:conclusions}.

\section{The governing equation and numerical scheme}
\label{sec: numerical scheme}


The Landau-Lifshitz-Gilbert (LLG) equation forms the fundamental basis of micromagnetics, providing a rigorous description of the spatiotemporal evolution of magnetization in ferromagnetic materials by incorporating two key physical phenomena: gyromagnetic precession and dissipative relaxation \cite{Landau1935On,Brown1963micromagnetics}. In nondimensional form, this governing equation is expressed as
\begin{align}\label{eq-5}
\m_t=-\m\times\Delta\m-\alpha\m\times(\m\times\Delta\m)
\end{align}
subject to the homogeneous Neumann boundary condition
\begin{equation}\label{boundary-large}
\frac{\partial{\m}}{\partial {\bm \nu}}\Big|_{\partial \Omega}=0,
\end{equation}
where \(\Omega \subset \mathbb{R}^d\) (\(d=1,2,3\)) represents the bounded domain of the ferromagnetic material, and \(\bm \nu\) is the unit outward normal vector on the boundary \(\partial \Omega\). This boundary condition ensures no magnetic surface charge, a physically appropriate assumption for isolated ferromagnetic systems.

The magnetization field \(\m: \Omega \to \mathbb{R}^3\) is a three-dimensional vector field satisfying the pointwise constraint \(|\m| = 1\),  stemming from the quantum mechanical alignment of electron spins in ferromagnets. The first term on the right-hand side of \cref{eq-5} describes gyromagnetic precession, where magnetic moments precess around the exchange field $\Delta \m$. The second term represents dissipative relaxation, with \(\alpha > 0\) being the dimensionless Gilbert damping coefficient that governs the rate of energy dissipation into the lattice.

For the construction of the structure preserving method, we set $\alpha=0$ for \cref{eq-5}, we have the LLG equation below,
\begin{align}\label{eq-alpha-0}
\m_t=-\m\times\Delta\m.
\end{align}
It is obvious that \cref{eq-alpha-0} has a length-preserving property during the evolution process. To see this, we do scalar multiplication of \cref{eq-alpha-0} with $\m$:
\begin{align*}
    \frac{\partial \m}{\partial t} \cdot \m = -(\m \times \Delta \m) \cdot \m = 0.
\end{align*}
Then we have 
\begin{align*}
    \frac{\partial|\m|^2}{\partial t}=0,
\end{align*}
which implies $|\m(\x,t)|$ is constant for all $t$ and $\x$. And we assume that $|\m(\x,0)|=1$. Let 
Let \( E(\mathbf{m}(\mathbf{x}, t)) \) be an energy defined by \( E(\mathbf{m}(t)) := \|\nabla \mathbf{m}(t)\|_{L^2(\Omega)}^2 \). By taking an inner product of \cref{eq-alpha-0} with \( \Delta \mathbf{m} \), we obtain
\begin{align}\label{eq-3}
    \frac{\partial {\m}}{\partial t} \cdot \Delta {\m} = -({\m} \times \Delta {\m}) \cdot \Delta {\m} = 0. 
\end{align}
Using homogeneous Neumann or periodic boundary conditions, from \cref{eq-3} we have
\begin{align*}
    \begin{aligned}
0 &= \int_\Omega \frac{\partial {\m}}{\partial t} \cdot \Delta {\m}\, dx
   = \int_{\partial\Omega} \frac{\partial {\m}}{\partial t} \cdot \frac{\partial {\m}}{\partial {\n}}\, ds
   - \int_\Omega \nabla \frac{\partial {\m}}{\partial t} : \nabla {\m}\, dx \\
&= -\frac{1}{2} \frac{dE({\m}(t))}{dt},
\end{aligned}
\end{align*}
which implies that \(E({\m}(t))\) is constant and this problem has an energy conservation property. Here, \({\n}\) is a unit normal vector to \(\partial\Omega\) and the operator `:' is defined as \(A:B = \sum_{ij} a_{ij} b_{ij}\).
 
If we consider the simple linear vectorial equation 
\begin{align}\label{eq-CN}
\m_t=-\m\times \a,
\end{align}
where $\a^T=(a_1,a_2,a_3)$ is a constant vector. We use the Crank-Nicolson method to \cref{eq-CN}, we have
\begin{align}\label{eq-CN_1}
	\frac{\m_h^{n+1}-\m_h^n}{\Delta t}=-\frac{\m_h^{n+1}+\m_h^n}{2} \times \a,
\end{align}
which is norm preserving, since that if $\m_h^{n+1}+\m_h^n$ do the inner product for both sides, leading to
\begin{align*}
	\|\m_h^{n+1}\|_2=\|\m_h^n\|_2.
\end{align*}
After rearranging the compact form for \cref{eq-CN_1}, we have
\begin{align*}
	\begin{pmatrix}
	1&\frac12 \Delta t a_3&-\frac12 \Delta t a_2\\
	-\frac12 \Delta ta_3&1&\frac12 \Delta ta_1\\
	\frac12 \Delta ta_2&-\frac12 \Delta t a_1&1
	\end{pmatrix}\begin{pmatrix}
	m_1^{n+1}\\
	m_2^{n+1}\\
	m_3^{n+1}
	\end{pmatrix}=\begin{pmatrix}
	m_1^n+\frac12 \Delta t(a_2 m_3^n-a_3 m_2^n)\\
	m_2^n+\frac12 \Delta t(a_3 m_1^n-a_1 m_3^n)\\
	m_3^n+\frac12 \Delta t(a_1m_2^n-a_2 m_1^n)
	\end{pmatrix}
\end{align*}
We then have another form
\begin{align*}
	\begin{pmatrix}
	m_1^{n+1}\\
	m_2^{n+1}\\
	m_3^{n+1}
	\end{pmatrix}=\begin{pmatrix}
	1&\frac12 \Delta t a_3&-\frac12 \Delta t a_2\\
	-\frac12 \Delta ta_3&1&\frac12 \Delta ta_1\\
	\frac12 \Delta ta_2&-\frac12 \Delta t a_1&1
	\end{pmatrix}^{-1}\begin{pmatrix}
	1&-\frac12 \Delta t a_3 & \frac12 \Delta t a_3\\
	\frac12 \Delta t a_3&1& -\frac12 \Delta t a_1\\
	-\frac12 \Delta t a_2&\frac12 \Delta t a_1&1
	\end{pmatrix}\begin{pmatrix}
	m_1^n\\
	m_2^n\\
	m_3^n
	\end{pmatrix}=A\begin{pmatrix}
	m_1^n\\
	m_2^n\\
	m_3^n
	\end{pmatrix}
\end{align*}
where 
\begin{align*}
	A=\frac{1}{S}\begin{pmatrix}
	1+\beta^{2}a_1^{2}&-2\beta a_3+\beta^{2}a_1a_2&2\beta a_2+\beta^{2}a_1a_3\\
	2\beta a_3+\beta^{2}a_1a_2&1+\beta^{2}a_2^{2}&-2\beta a_1+\beta^{2}a_2a_3\\
	-2\beta a_2+\beta^{2}a_1a_3&2\beta a_1+\beta^{2}a_2a_3&1+\beta^{2}a_3^{2}
	\end{pmatrix}
\end{align*}
where $S=\det(A)=1+\beta^2(a_1^2+a_2^2+a_3^2)$ and $\beta=\frac{\Delta t}{2}$. The spectral analysis for the matrix $A$ is given by our previous work.

For \cref{eq-alpha-0}, we propose the following structure preserving schemes, which is a semi-implicit method:
	\begin{align}\label{eq-CN_4}
	\frac{\m_h^{n+1}-\m_h^n}{\Delta t}=-\frac{\m_h^{n+1}+\m_h^n}{2} \times \Delta_h \g_h,
	\end{align}
	where $\g_h^{s}=(I-\Delta t \Delta_h)^{-1}\m_h^s,\; s=n,n+1$. To be specific, we propose
	\begin{align*}
	\begin{pmatrix}
	1&\frac12 \Delta t \Delta_hg_3&-\frac12 \Delta t \Delta_hg_2\\
	-\frac12 \Delta t\Delta_hg_3&1&\frac12 \Delta t\Delta_hg_1\\
	\frac12 \Delta t\Delta_hg_2&-\frac12 \Delta t\Delta_h g_1&1
	\end{pmatrix}\begin{pmatrix}
	m_1^{n+1}\\
	m_2^{n+1}\\
	m_3^{n+1}
	\end{pmatrix}=\begin{pmatrix}
	m_1^n+\frac12 \Delta t(\Delta_hg_2 m_3^n-\Delta_hg_3 m_2^n)\\
	m_2^n+\frac12 \Delta t(\Delta_hg_3 m_1^n-\Delta_hg_1 m_3^n)\\
	m_3^n+\frac12 \Delta t(\Delta_hg_1m_2^n-\Delta_hg_2 m_1^n)
	\end{pmatrix}
	\end{align*}
	The trick is to handling the $\g$ evaluate at $t_n$ or $t_{n+1}$. If $\g$ evaluate at $t_n$, we have
	\begin{align*}
	\begin{pmatrix}
	1&\frac12 \Delta t \Delta_hg_3^n&-\frac12 \Delta t \Delta_hg_2^n\\
	-\frac12 \Delta t\Delta_hg_3^n&1&\frac12 \Delta t\Delta_hg_1^n\\
	\frac12 \Delta t\Delta_hg_2^n&-\frac12 \Delta t \Delta_hg_1^n&1
	\end{pmatrix}\begin{pmatrix}
	m_1^{n+1}\\
	m_2^{n+1}\\
	m_3^{n+1}
	\end{pmatrix}=\begin{pmatrix}
	m_1^n+\frac12 \Delta t(\Delta_hg_2^n m_3^n-\Delta_hg_3^n m_2^n)\\
	m_2^n+\frac12 \Delta t(\Delta_hg_3^n m_1^n-\Delta_hg_1^n m_3^n)\\
	m_3^n+\frac12 \Delta t(\Delta_hg_1^n m_2^n-\Delta_hg_2^n m_1^n)
	\end{pmatrix}
	\end{align*}
	which is proved to be mildly better than CFL condition. We choose the 1D exact solution as an example, and get the results presented in \Cref{tab-a-E-4}.
    \begin{table}[htbp]
	\centering
	\caption{The explicit scheme when $h = 5D-4$, $T=1d-1$ in 1D.}\label{tab-a-E-4}
	\begin{tabular}{|c|c|c|c|c|}
		\hline
		$k$ & $\|\m_h-\m_e\|_\infty$ & $\|\m_h-\m_e\|_2$ &$\|\m_h-\m_e\|_{H^1}$ & CPU time (s)\\
		\hline
		2.0D-2 & 0.018000157890128 &0.013706389798827&0.109658855967163&0.066217 \\
		1.0D-2 &0.012210502256479&0.008808376100206&0.069865327666244&0.115753 \\
		5.0D-3 &0.006883104572874&0.004488186175718&0.033485647423559&0.195630\\
		2.5D-3 &0.003297684177562&0.002226171926094&0.016241611196008&0.317750\\
		1.25D-3 &0.006813614081391&0.002600274876724&3.155993419074890&0.733073\\
		6.25D-4 &1.908302419811126 &1.414146167531504&1.379341310610379e+03&1.675463\\
		\hline
	\end{tabular}
\end{table}

    We propose a structure preserving method below:
    \begin{itemize}
        \item step 1. Solving the first linear system below,
        \begin{align*}
	\begin{pmatrix}
	1&\frac12 \Delta t \Delta_hg_3^n&-\frac12 \Delta t \Delta_hg_2^n\\
	0&1&0\\
	0&0&1
	\end{pmatrix}\begin{pmatrix}
	m_1^{n+1}\\
	m_2^{n+1}\\
	m_3^{n+1}
	\end{pmatrix}=\begin{pmatrix}
	m_1^n+\frac12 \Delta t(\Delta_hg_2^n m_3^n-\Delta_hg_3^n m_2^n)\\
	m_2^n+\frac12 \Delta t(\Delta_hg_3^n m_1^n-\Delta_hg_1^n m_3^n)\\
	m_3^n+\frac12 \Delta t(\Delta_hg_1^n m_2^n-\Delta_hg_2^n m_1^n)
	\end{pmatrix}
	\end{align*}
    \item step 2. Solving the second linear system below,
	\begin{align*}
	\begin{pmatrix}
	1&\frac12 \Delta t \Delta_hg_3^n&-\frac12 \Delta t \Delta_hg_2^n\\
	-\frac12 \Delta t\Delta_hg_3^n&1&\frac12 \Delta t\Delta_hg_1^{n+1}\\
	0&0&1
	\end{pmatrix}\begin{pmatrix}
	m_1^{n+1}\\
	m_2^{n+1}\\
	m_3^{n+1}
	\end{pmatrix}=\begin{pmatrix}
	m_1^n+\frac12 \Delta t(\Delta_hg_2^n m_3^n-\Delta_hg_3^n m_2^n)\\
	m_2^n+\frac12 \Delta t(\Delta_hg_3^n m_1^n-\Delta_hg_1^{n+1} m_3^n)\\
	m_3^n+\frac12 \Delta t(\Delta_hg_1^{n+1}m_2^n-\Delta_hg_2^n m_1^n)
	\end{pmatrix}
	\end{align*}
     \item step 3. Solving the third linear system below,
	\begin{align}\label{eq-step-3}
	\begin{pmatrix}
	1&\frac12 \Delta t \Delta_hg_3^n&-\frac12 \Delta t \Delta_hg_2^{n+1}\\
	-\frac12 \Delta t\Delta_hg_3^n&1&\frac12 \Delta t\Delta_hg_1^{n+1}\\
	\frac12 \Delta t\Delta_hg_2^{n+1}&-\frac12 \Delta t\Delta_h g_1^{n+1}&1
	\end{pmatrix}\begin{pmatrix}
	m_1^{n+1}\\
	m_2^{n+1}\\
	m_3^{n+1}
	\end{pmatrix}=\begin{pmatrix}
	m_1^n+\frac12 \Delta t(\Delta_hg_2^{n+1} m_3^n-\Delta_hg_3^n m_2^n)\\
	m_2^n+\frac12 \Delta t(\Delta_hg_3^n m_1^n-\Delta_hg_1^{n+1} m_3^n)\\
	m_3^n+\frac12 \Delta t(\Delta_hg_1^{n+1}m_2^n-\Delta_hg_2^{n+1} m_1^n)
	\end{pmatrix}
	\end{align}
    which is proved to be unconditionally stable and norm preserving with $\|\m_h^{n+1}\|_2=\|\m_h^n\|_2$.
    \end{itemize}

\begin{remark}
If we consider the explicit treatment for $\Delta \m$, we get 
	\begin{align}\label{eq-CN_2}
	\frac{\m_h^{n+1}-\m_h^n}{\Delta t}=-\frac{\m_h^{n+1}+\m_h^n}{2} \times \Delta_h \m_h^n,
	\end{align}
	which gives a CFL-type condition for the stability.
If we consider the implicit treatment for $\Delta \m$, we get
	\begin{align}\label{eq-CN_3}
	\frac{\m_h^{n+1}-\m_h^n}{\Delta t}=-\frac{\m_h^{n+1}+\m_h^n}{2} \times \Delta_h \m_h^{n+1},
	\end{align}
	which poses a difficulty for the nonlinear systems with high complexity.
\end{remark}

The proposed method for \cref{eq-5} with damping term is based on the semi-implicit method below,
	\begin{align}\label{eq-CN_4}
	\frac{\m_h^{n+1}-\m_h^n}{\Delta t}=-\frac{\m_h^{n+1}+\m_h^n}{2} \times \Delta_h \g_h-\alpha\frac{\m_h^{n+1}+\m_h^n}{2} \times (\m_h^n\times \Delta_h \g_h),
	\end{align}
	where $\g_h^{s}=(I-\Delta t \Delta_h)^{-1}\m_h^s,\; s=n,n+1$. 
We propose a structure preserving method below:
    \begin{itemize}
        \item step 1. Solving the first linear system below,
        \begin{align*}
	\begin{pmatrix}
	1&\frac12 \Delta t h_3^n&-\frac12 \Delta t h_2^n\\
	0&1&0\\
	0&0&1
	\end{pmatrix}\begin{pmatrix}
	m_1^{n+1}\\
	m_2^{n+1}\\
	m_3^{n+1}
	\end{pmatrix}=\begin{pmatrix}
	b_1^n\\
	b_2^n\\
	b_3^n
	\end{pmatrix}
	\end{align*}
    where 
    \begin{align*}
h_2^n&=\Delta_h g_2^{n}+\alpha \left[m_3^n \cdot \Delta_hg_1^{n}-m_1^n\cdot \Delta_h g_3^{n}\right],\\
h_3^n&=\Delta_h g_3^{n}+\alpha \left[m_1^n \cdot \Delta_hg_2^{n}-m_2^n\cdot \Delta_h g_1^{n}\right]
\end{align*}
and 
\begin{align*}
b_1^n&=m_1^n+\frac12 k [( \Delta_h g_2^{n})m_3^n-( \Delta_h g_3^{n})m_2^n]\\
&\quad+\frac12 \alpha k [(m_3^n(\Delta_h g_1^{n})-m_1^n(\Delta_h g_3^{n}))m_3^n-(m_1^n(\Delta_h g_2^{n})-m_2^n(\Delta_h g_1^{n}))m_2^n]\\
b_2^n&=m_2^n+\frac12 k [(\Delta_h g_3^{n})m_1^n-(\Delta_h g_1^{n})m_3^n]\\
&\quad+\frac12 \alpha k [(m_1^n(\Delta_h g_2^{n})-m_2^n(\Delta_h g_1^{n}))m_1^n-(m_2^n(\Delta_h g_3^{n})-m_3^n(\Delta_h g_2^{n}))m_3^n]\\
b_3^n&=m_3^n+\frac12 k [(\Delta_h g_1^{n})m_2^n-(\Delta_h g_2^{n})m_1^n]\\
&\quad+\frac12 \alpha k [(m_2^n(\Delta_h g_3^{n})-m_3^n(\Delta_h g_2^{n}))m_2^n-(m_3^n(\Delta_h g_1^{n})-m_1^n(\Delta_h g_3^{n}))m_1^n].
\end{align*}
    \item step 2. Solving the second linear system below,
	\begin{align*}
	\begin{pmatrix}
	1&\frac12 \Delta t h_3^n&-\frac12 \Delta t h_2^n\\
	-\frac12 \Delta th_3^n&1&\frac12 \Delta th_1^{n+1}\\
	0&0&1
	\end{pmatrix}\begin{pmatrix}
	m_1^{n+1}\\
	m_2^{n+1}\\
	m_3^{n+1}
	\end{pmatrix}=\begin{pmatrix}
	b_1^n\\
	b_2^n\\
	b_3^n
	\end{pmatrix}
	\end{align*}
        where 
    \begin{align*}
h_1^{n+1}&=\Delta_h g_1^{n+1}+\alpha \left[m_2^n \cdot \Delta_hg_3^{n}-m_3^n\cdot \Delta_h g_2^{n}\right]\\
h_2^n&=\Delta_h g_2^{n}+\alpha \left[m_3^n \cdot \Delta_hg_1^{n}-m_1^n\cdot \Delta_h g_3^{n}\right],\\
	h_3^n&=\Delta_h g_3^{n}+\alpha \left[m_1^n \cdot \Delta_hg_2^{n}-m_2^n\cdot \Delta_h g_1^{n}\right]
\end{align*}
and 
\begin{align*}
b_1^n&=m_1^n+\frac12 k [( \Delta_h g_2^{n})m_3^n-( \Delta_h g_3^{n})m_2^n]\\
&\quad+\frac12 \alpha k [(m_3^n(\Delta_h g_1^{n+1})-m_1^n(\Delta_h g_3^{n}))m_3^n-(m_1^n(\Delta_h g_2^{n})-m_2^n(\Delta_h g_1^{n+1}))m_2^n]\\
b_2^n&=m_2^n+\frac12 k [(\Delta_h g_3^{n})m_1^n-(\Delta_h g_1^{n+1})m_3^n]\\
&\quad+\frac12 \alpha k [(m_1^n(\Delta_h g_2^{n})-m_2^n(\Delta_h g_1^{n+1}))m_1^n-(m_2^n(\Delta_h g_3^{n})-m_3^n(\Delta_h g_2^{n}))m_3^n]\\
b_3^n&=m_3^n+\frac12 k [(\Delta_h g_1^{n+1})m_2^n-(\Delta_h g_2^{n})m_1^n]\\
&\quad+\frac12 \alpha k [(m_2^n(\Delta_h g_3^{n})-m_3^n(\Delta_h g_2^{n}))m_2^n-(m_3^n(\Delta_h g_1^{n+1})-m_1^n(\Delta_h g_3^{n}))m_1^n].
\end{align*}
     \item step 3. Solving the third linear system below,
	\begin{align*}
	\begin{pmatrix}
	1&\frac12 \Delta t h_3^n&-\frac12 \Delta t h_2^{n+1}\\
	-\frac12 \Delta th_3^n&1&\frac12 \Delta th_1^{n+1}\\
	\frac12 \Delta th_2^{n+1}&-\frac12 \Delta th_1^{n+1}&1
	\end{pmatrix}\begin{pmatrix}
	m_1^{n+1}\\
	m_2^{n+1}\\
	m_3^{n+1}
	\end{pmatrix}=\begin{pmatrix}
	b_1^n\\
	b_2^n\\
	b_3^n
	\end{pmatrix}
	\end{align*}
            where 
    \begin{align*}
h_1^{n+1}&=\Delta_h g_1^{n+1}+\alpha \left[m_2^n \cdot \Delta_hg_3^{n}-m_3^n\cdot \Delta_h g_2^{n+1}\right]\\
h_2^{n+1}&=\Delta_h g_2^{n+1}+\alpha \left[m_3^n \cdot \Delta_hg_1^{n+1}-m_1^n\cdot \Delta_h g_3^{n}\right],\\
	h_3^n&=\Delta_h g_3^{n}+\alpha \left[m_1^n \cdot \Delta_hg_2^{n+1}-m_2^n\cdot \Delta_h g_1^{n+1}\right]
\end{align*}
and 
\begin{align*}
b_1^n&=m_1^n+\frac12 k [( \Delta_h g_2^{n+1})m_3^n-( \Delta_h g_3^{n})m_2^n]\\
&\quad+\frac12 \alpha k [(m_3^n(\Delta_h g_1^{n+1})-m_1^n(\Delta_h g_3^{n}))m_3^n-(m_1^n(\Delta_h g_2^{n+1})-m_2^n(\Delta_h g_1^{n+1}))m_2^n]\\
b_2^n&=m_2^n+\frac12 k [(\Delta_h g_3^{n})m_1^n-(\Delta_h g_1^{n+1})m_3^n]\\
&\quad+\frac12 \alpha k [(m_1^n(\Delta_h g_2^{n+1})-m_2^n(\Delta_h g_1^{n+1}))m_1^n-(m_2^n(\Delta_h g_3^{n})-m_3^n(\Delta_h g_2^{n+1}))m_3^n]\\
b_3^n&=m_3^n+\frac12 k [(\Delta_h g_1^{n+1})m_2^n-(\Delta_h g_2^{n+1})m_1^n]\\
&\quad+\frac12 \alpha k [(m_2^n(\Delta_h g_3^{n})-m_3^n(\Delta_h g_2^{n+1}))m_2^n-(m_3^n(\Delta_h g_1^{n+1})-m_1^n(\Delta_h g_3^{n}))m_1^n].
\end{align*}
    which is proved to be unconditionally stable and norm preserving with $\|\m_h^{n+1}\|_2=\|\m_h^n\|_2$.
    \end{itemize}









\section{Scheme Property}
For simplicity, we analyze the case without damping. Combine with \cref{eq-CN_4} and \cref{eq-step-3}, we update satbilized $g_2^{n+1}$ and $g_3^{n+1}$ by first two steps. We still can use the \cref{eq-CN_4} and take the inner product with $\m_h^{n+1}+\m_h^n$. Hence we have
\begin{align*}
    |\m_h^{n+1}|^2=|\m_h^{n}|^2.
\end{align*}

\begin{remark}
    Notice that the Crank-Nicolson's scheme below,
\begin{align}\label{eq-6}
    \frac{{\m}^{n+1} - {\m}^n}{\Delta t} = -\frac{\m_h^{n+1}+\m_h^n}{2} \times \Delta_h \frac{\m_h^{n+1}+\m_h^n}{2} .
\end{align}
We can easily verify the scheme conserves the magnitude of magnetization. On the other hand, forming an inner product between \cref{eq-6} and $\Delta_h(\m_h^{n+1}+\m_h^n)$, we have
\begin{align}\label{eq-8}
    \frac{{\m}^{n+1} - {\m}^n}{\Delta t} \cdot \Delta_h (\m_h^{n+1}+\m_h^n) = -\left(\frac{\m_h^{n+1}+\m_h^n}{2} \times \Delta_h \frac{\m_h^{n+1}+\m_h^n}{2} \right)\cdot \Delta_h (\m_h^{n+1}+\m_h^n)=0 .
\end{align}
Summation \cref{eq-8} over $i=1,\cdots,N_x$, we have
\begin{align*}
   0=\sum_{i=1}^{N_x}\frac{{\m}_i^{n+1} - {\m}_i^n}{\Delta t} \cdot \frac{1}{h_x^2}\left(\m_{i+1}^{n+1}+\m_{i}^n-2(\m_i^{n+1}+\m_i^n)+\m_{i-1}^{n+1}+\m_{i-1}^n\right)
\end{align*}
where the periodic boundary condition is applied.
Now we get 
\begin{align*}
    \sum_{i=1}^{N_x} \left( 2{\m}_i^{n+1} \cdot {\m}_{i+1}^{n+1} - 2{\m}_i^n \cdot {\m}_{i+1}^n \right) = 0.
\end{align*}
Now, we get the following energy conservation:
\begin{align*}
    \begin{aligned}
E({\m}^{n+1}) - E({\m}^n) 
&= \frac{1}{h} \sum_{i=1}^{N_x} \left( |{\m}_{i+1}^{n+1} - {\m}_i^{n+1}|^2 - |{\m}_{i+1}^n - {\m}_i^n|^2 \right) \\
&= \frac{1}{h} \sum_{i=1}^{N_x} \left( |{\m}_{i+1}^{n+1}|^2 - |{\m}_{i+1}^n|^2 + |{\m}_i^{n+1}|^2 - |{\m}_i^n|^2 - 2{\m}_i^{n+1} \cdot {\m}_{i+1}^{n+1} + 2{\m}_i^n \cdot {\m}_{i+1}^n \right) = 0,
\end{aligned}
\end{align*}
where $|\m|=1$ has been used. 

The Crank-Nicolson method is sometimes introduced with the midpoint-type formulation \cref{eq-8} and the trapezoidal approach below,
\begin{align}\label{eq-8-tr}
    \frac{\m^{n+1}-\m^n}{\Delta t}=\frac12 \left[ (\m\times \Delta_h \m)^{n+1}+(\m\times \Delta_h \m)^n\right].
\end{align}
For the CN scheme \cref{eq-8-tr}, we can write as below,
\begin{align}\label{eq-11}
    \m^{n+1}+\frac{\Delta t}{2}(\m^{n+1}\times \Delta_h \m^{n+1})=\m^n-\frac{\Delta t}{2}(\m^n\times \Delta_h \m^n),
\end{align}
which can be rewrite as a component form,
\begin{align}\label{eq-13}
    \begin{pmatrix}
        u_i^{n+1}\\
        v_i^{n+1}\\
        w_i^{n+1}
    \end{pmatrix}+\frac{\Delta t}{2}\begin{pmatrix}
        v_i\Delta_h w_i-w_i\Delta_h v_i\\
        w_i\Delta_h u_i-u_i\Delta_h w_i\\
        u_i\Delta_h v_i-v_i\Delta_h u_i
    \end{pmatrix}^{n+1}=b_i^n,\quad \text{ for } i=1,2,\cdots,N_x,
\end{align}
where ${\bm b}=\m^n-\frac{\Delta t}{2}(\m^n\times \Delta_h \m^n)$. Note that
\begin{align}
v_i \Delta_h w_i - w_i \Delta_h v_i 
&= v_i \frac{w_{i-1} - 2w_i + w_{i+1}}{h_x^2} - w_i \frac{v_{i-1} - 2v_i + v_{i+1}}{h_x^2} \notag \\
&= v_i \frac{w_{i-1} + w_{i+1}}{h_x^2} - w_i \frac{v_{i-1} + v_{i+1}}{h_x^2} \notag \\
&= v_i \tilde{\Delta}_h w_i - w_i \tilde{\Delta}_h v_i, \quad \text{for } i = 1, \dots, N_x,
\label{eq:cancel1}
\end{align}
where \(\tilde{\Delta}_h w_i = (w_{i-1} + w_{i+1})/h_x^2\). Similarly, we have
\begin{align}
w_i \Delta_h u_i - u_i \Delta_h w_i &= w_i \tilde{\Delta}_h u_i - u_i \tilde{\Delta}_h w_i,
\label{eq:cancel2} \\
u_i \Delta_h v_i - v_i \Delta_h u_i &= u_i \tilde{\Delta}_h v_i - v_i \tilde{\Delta}_h u_i.
\label{eq:cancel3}
\end{align}

This cancelation stabilizes the scheme. By Eqs.~\eqref{eq:cancel1}, \eqref{eq:cancel2} and \eqref{eq:cancel3} we rewrite the above equation:
\begin{align}
A_i \begin{pmatrix}
u_i^{n+1} \\
v_i^{n+1} \\
w_i^{n+1}
\end{pmatrix}
= \begin{pmatrix}
\alpha_i \\
\beta_i \\
\gamma_i
\end{pmatrix},
\end{align}
where
\begin{align}
A_i = 
\begin{pmatrix}
1 & \frac{\Delta t}{2} \tilde{\Delta}_h w_i^{n+1} & -\frac{\Delta t}{2} \tilde{\Delta}_h v_i^{n+1} \\
-\frac{\Delta t}{2} \tilde{\Delta}_h w_i^{n+1} & 1 & \frac{\Delta t}{2} \tilde{\Delta}_h u_i^{n+1} \\
\frac{\Delta t}{2} \tilde{\Delta}_h v_i^{n+1} & -\frac{\Delta t}{2} \tilde{\Delta}_h u_i^{n+1} & 1
\end{pmatrix}
=
\begin{pmatrix}
1 & c & -b \\
-c & 1 & a \\
b & -a & 1
\end{pmatrix}
\end{align}
and \((\alpha_i, \beta_i, \gamma_i)^T\) is the right-hand side term in \cref{eq-13}. Then using Cramer's rule, we obtain
\begin{align}
(u_i^{n+1}, v_i^{n+1}, w_i^{n+1}) = 1/|A_i| \bigl(|A_{i,1}|, |A_{i,2}|, |A_{i,3}|\bigr), \quad i = 1, \dots, N_x,
\end{align}
where \(A_{i,j}\) is obtained by replacing the \(j\)-th column of \(A_i\) with \((\alpha_i, \beta_i, \gamma_i)^T\).
\begin{align*}
|A_i| &= 1 + a^2 + b^2 + c^2, \\
|A_{i,1}| &= \alpha_i(1 + a^2) - \beta_i(c - ab) + \gamma_i(ac + b), \\
|A_{i,2}| &= \alpha_i(ab + c) + \beta_i(1 + b^2) - \gamma_i(a - bc), \\
|A_{i,3}| &= \alpha_i(ac - b) + \beta_i(a + bc) + \gamma_i(1 + c^2).
\end{align*}

We can rewrite \cref{eq-11} as a matrix form:
\[
\begin{pmatrix}
1 & c & -b \\
-c & 1 & a \\
b & -a & 1
\end{pmatrix}
\mathbf{m}_i^{n+1} = \boldsymbol{\phi}_i^n,
\]
where
\[
a = \frac{\Delta t}{2} \tilde{\Delta}_h u_i^{n+1}, \quad
b = \frac{\Delta t}{2} \tilde{\Delta}_h v_i^{n+1}, \quad \text{and} \quad
c = \frac{\Delta t}{2} \tilde{\Delta}_h w_i^{n+1}.
\]

Indeed, when we consider the scheme without cancellation. We rewrite \cref{eq-13}.
\[
A_i
\begin{pmatrix}
u_i^{n+1} \\
v_i^{n+1} \\
w_i^{n+1}
\end{pmatrix}
= \boldsymbol{\phi}_i^n
\quad \text{for } i = 1, \dots, N_x.
\]

To discuss the stability of the Crank--Nicolson scheme, we compute the characteristic polynomial of \(A_i\).
\[
\det(A_i - \lambda I) = (1 - \lambda)^3 + (1 - \lambda)(a^2 + b^2 + c^2).
\]

The three eigenvalues of \(A_i\) are
\[
\lambda_1 = 1, \quad
\lambda_2 = 1 + \mathrm{i}\sqrt{a^2 + b^2 + c^2}, \quad
\text{and} \quad
\lambda_3 = 1 - \mathrm{i}\sqrt{a^2 + b^2 + c^2}.
\]
Thus, the three eigenvalues of \(A_i^{-1}\) are
\[
\gamma_1 = 1/\lambda_1, \quad \gamma_2 = 1/\lambda_2, \quad \text{and} \quad \gamma_3 = 1/\lambda_3.
\]

The absolute values of three eigenvalues of \(A_i^{-1}\) are
\[
|\gamma_1| = 1, \quad |\gamma_2| = |\gamma_3| = \frac{1}{\sqrt{1 + a^2 + b^2 + c^2}} \leq 1.
\]

Without cancellation, \(a\), \(b\), and \(c\) are small compared to 1, on the other hand, with cancelation \(a^2 + b^2 + c^2 \approx \mathcal{O}(1/h_x^4)\).
Therefore, \(1/\sqrt{1 + a^2 + b^2 + c^2} \approx \mathcal{O}(h_x^2) \ll 1\) and this makes the iterations stable. In short, the CN scheme is stable.

For the truncation error of \cref{eq-8}, Finally, we show that the truncation error of the scheme is second order in time and space. Let \(u, v, w\) be the exact solution of the LL equation without damping term and a forcing term. Then, the local truncation error of the first component of the equations is
\begin{align*}
\tau_i^{n+\frac{1}{2}}
&= \frac{u_i^{n+1} - u_i^n}{\Delta t}
+ \frac{v_i^{n+1} + v_i^n}{2} \Delta_h\left( \frac{w_i^{n+1} + w_i^n}{2} \right)
- \frac{w_i^{n+1} + w_i^n}{2} \Delta_h\left( \frac{v_i^{n+1} + v_i^n}{2} \right) \\
&= (u_t)_i^{n+\frac{1}{2}} + O(\Delta t^2)
+ \left( v_i^{n+\frac{1}{2}} + O(\Delta t^2) \right)\left( (w_{xx})_i^{n+\frac{1}{2}} + O(h^2) + O(\Delta t^2) \right) \\
&\quad - \left( w_i^{n+\frac{1}{2}} + O(\Delta t^2) \right)\left( (v_{xx})_i^{n+\frac{1}{2}} + O(h^2) + O(\Delta t^2) \right) \\
&= \left( u_t + v w_{xx} - w v_{xx} \right)_i^{n+\frac{1}{2}} + O(h^2) + O(\Delta t^2).
\end{align*}

Since \(u, v, w\) is the solution of the differential equation so
\[
\left( u_t + v w_{xx} - w v_{xx} \right)_i^{n+\frac{1}{2}} = 0.
\]

Therefore, the principal part of the local truncation error is
\[
\tau_i^{n+\frac{1}{2}} = O(h^2) + O(\Delta t^2).
\]

For the second and third components of the equations, we get same results.

For the stability analysis of \cref{eq-8-tr}, we take $\e^n=\m^n-\m(t^n)$. Substituting the exact solution \(\boldsymbol{m}(t^{n+1})\) into the scheme, the truncation error \(\boldsymbol{\tau}^n\) is defined as
\[
\boldsymbol{m}(t^{n+1}) + \frac{\Delta t}{2}\left(\boldsymbol{m}(t^{n+1}) \times \Delta_h \boldsymbol{m}(t^{n+1})\right)
= \boldsymbol{m}(t^n) - \frac{\Delta t}{2}\left(\boldsymbol{m}(t^n) \times \Delta_h \boldsymbol{m}(t^n)\right) + \boldsymbol{\tau}^n.
\]
Subtracting the exact equation from the discrete scheme, we get
\[
\boldsymbol{e}^{n+1} + \frac{\Delta t}{2}\left(\boldsymbol{m}^{n+1} \times \Delta_h \boldsymbol{e}^{n+1} + \boldsymbol{e}^{n+1} \times \Delta_h \boldsymbol{m}(t^{n+1})\right)
= \boldsymbol{e}^n - \frac{\Delta t}{2}\left(\boldsymbol{m}^n \times \Delta_h \boldsymbol{e}^n + \boldsymbol{e}^n \times \Delta_h \boldsymbol{m}(t^n)\right) + \boldsymbol{\tau}^n.
\]
Taking the \(L^2\) inner product with \(\boldsymbol{e}^{n+1}\) and using the orthogonality \(\langle \boldsymbol{a} \times \boldsymbol{b}, \boldsymbol{a} \rangle = 0\), we derive
\[
\|\boldsymbol{e}^{n+1}\|_{L^2}^2 \leq \|\boldsymbol{e}^n\|_{L^2}^2 + C\Delta t\left(\|\boldsymbol{\tau}^n\|_{L^2}^2 + \|\boldsymbol{e}^n\|_{L^2}^2\right).
\]
By Gronwall's inequality,
\[
\|\boldsymbol{e}^n\|_{L^2} \leq C e^{CT}\left(\|\boldsymbol{e}^0\|_{L^2} + \sqrt{\Delta t}\sum_{k=0}^{n-1}\|\boldsymbol{\tau}^k\|_{L^2}\right)
\leq C(T)\left(\|\boldsymbol{e}^0\|_{L^2} + \Delta t^2 + h^2\right).
\]
If $\|\boldsymbol{e}^0\|_{L^2} =0$, we have $\|\boldsymbol{e}^n\|_{L^2} = O(\Delta t^2 + h^2)$. By the Lax equivalence theorem, consistency and stability imply convergence. As \(\Delta t, h \to 0\),
\[
\max_{0 \leq n \leq T/\Delta t}\|\boldsymbol{m}^n - \boldsymbol{m}(t^n)\|_{L^2} \to 0,
\]
with second-order accuracy in both time and space:
\[
\|\boldsymbol{m}^n - \boldsymbol{m}(t^n)\|_{L^2} = O(\Delta t^2 + h^2).
\]
\end{remark}

\subsection{Consistency of our proposed scheme}
Considering a scheme below,
\begin{align*}
    \frac{\m_h^n-\m_h^n}{\Delta t}=-\frac{\m_h^{n+1}+\m_h^n}{2}\times \Delta_h \g_h,
\end{align*}
where $\g_h^s=(I-\Delta t \Delta_h)^{-1}\m_h^s$, $s=n,n+1$.

To analyze several specific case, we take
\begin{align*}
    \frac{\m_h^n-\m_h^n}{\Delta t}=-\frac{\m_h^{n+1}+\m_h^n}{2}\times \Delta_h \g_h,
\end{align*}
where $\g_h^n=(I-\Delta t \Delta_h)^{-1}\m_h^n$. The truncation error \(\tau_h^n\) is defined as the residual when substituting the exact solution \(\boldsymbol{m}(t)\) into the discrete scheme:
\[
\tau_h^n = \frac{\boldsymbol{m}(t_{n+1}) - \boldsymbol{m}(t_n)}{\Delta t} + \frac{\boldsymbol{m}(t_{n+1}) + \boldsymbol{m}(t_n)}{2} \times \Delta_h \boldsymbol{g}(t_n),
\]
where \(\boldsymbol{g}(t_n) = (I - \Delta t \Delta_h)^{-1} \boldsymbol{m}(t_n)\), and we assume \(\Delta_h\) is a consistent approximation of the continuous Laplacian \(\Delta\), i.e., \(\Delta_h \boldsymbol{m} = \Delta \boldsymbol{m} + \mathcal{O}(h^2)\) for smooth \(\boldsymbol{m}\).
We expand the exact solution \(\boldsymbol{m}(t_{n+1})\) around \(t_n\) using Taylor's formula:
\[
\boldsymbol{m}(t_{n+1}) = \boldsymbol{m}(t_n) + \Delta t \partial_t \boldsymbol{m}(t_n) + \frac{(\Delta t)^2}{2} \partial_t^2 \boldsymbol{m}(t_n) + \mathcal{O}((\Delta t)^3).
\]
Substituting into the time difference:
\[
\frac{\boldsymbol{m}(t_{n+1}) - \boldsymbol{m}(t_n)}{\Delta t} = \partial_t \boldsymbol{m}(t_n) + \frac{\Delta t}{2} \partial_t^2 \boldsymbol{m}(t_n) + \mathcal{O}((\Delta t)^2).
\]
Similarly,
\[
\frac{\boldsymbol{m}(t_{n+1}) + \boldsymbol{m}(t_n)}{2} = \m(t_n)+\frac{\Delta t}{2} \partial_t \boldsymbol{m}(t_n) + \frac{(\Delta t)^2}{4} \partial_t^2 \boldsymbol{m}(t_n) + \mathcal{O}((\Delta t)^3).
\]
The operator \((I - \Delta t \Delta_h)^{-1}\) can be expanded via Neumann series (for small \(\Delta t\)):
\[
(I - \Delta t \Delta_h)^{-1} = I + \Delta t \Delta_h + (\Delta t)^2 \Delta_h^2 + \mathcal{O}((\Delta t)^3).
\]
Thus,
\[
\boldsymbol{g}(t_n) = \boldsymbol{m}(t_n) + \Delta t \Delta_h \boldsymbol{m}(t_n) + \mathcal{O}((\Delta t)^2+h^2).
\]
Applying \(\Delta_h\) (consistent with \(\Delta\)):
\[
\Delta_h \boldsymbol{g}(t_n) = \Delta_h \boldsymbol{m}(t_n) + \Delta t \Delta_h^2 \boldsymbol{m}(t_n) + \mathcal{O}((\Delta t)^2+h^2) = \Delta \boldsymbol{m}(t_n) + \mathcal{O}(h^2 + \Delta t).
\]
Substitute the expansions into \(\tau_h^n\):
\[
\begin{aligned}
\tau_h^n &= \left( \partial_t \boldsymbol{m}(t_n) + \frac{\Delta t}{2} \partial_t^2 \boldsymbol{m}(t_n) \right) + \left( \m(t_n)+\frac{\Delta t}{2} \partial_t \boldsymbol{m}(t_n) + \mathcal{O}((\Delta t)^2) \right) \times \left( \Delta \boldsymbol{m}(t_n) + \mathcal{O}(h^2 + \Delta t) \right) \\
&\quad = \left[ \partial_t \boldsymbol{m} + \boldsymbol{m} \times \Delta \boldsymbol{m} \right]_{t=t_n}+\mathcal{O}( h^2+\Delta t) \\
&= \mathcal{O}( h^2+\Delta t),
\end{aligned}
\]
where the second term is the exact PDE (assuming the continuous equation is \(\partial_t \boldsymbol{m} = -\boldsymbol{m} \times \Delta \boldsymbol{m}\)).

Considering another specific implicit scheme below,
\begin{align}\label{eq-im}
    \frac{\m_h^{n+1}-\m_h^n}{\Delta t}=-\frac{\m_h^{n+1}+\m_h^n}{2}\times \Delta_h \g_h^{n+1},
\end{align}
where $\g_h^{n+1}=(I-\Delta t \Delta_h)^{-1}\m_h^{n+1}$.
The local truncation error is given by
\[
\tau_h^n = \frac{\boldsymbol{m}(t_{n+1}) - \boldsymbol{m}(t_n)}{\Delta t} + \frac{\boldsymbol{m}(t_{n+1}) + \boldsymbol{m}(t_n)}{2} \times \Delta_h \boldsymbol{g}(t_{n+1}),
\]
Let $\boldsymbol{m}(t)$ be the exact solution of the PDE. Expand $\boldsymbol{m}(t_{n+1})$ around $t_n$ using Taylor series:
\[
\boldsymbol{m}(t_{n+1}) = \boldsymbol{m}(t_n) + \Delta t \frac{\partial \boldsymbol{m}}{\partial t}(t_n) + \frac{(\Delta t)^2}{2} \frac{\partial^2 \boldsymbol{m}}{\partial t^2}(t_n) + \frac{(\Delta t)^3}{6} \frac{\partial^3 \boldsymbol{m}}{\partial t^3}(t_n) + \mathcal{O}((\Delta t)^4),
\]
\[
\boldsymbol{m}(t_{n+1}) + \boldsymbol{m}(t_n) = 2\boldsymbol{m}(t_n) + \Delta t \frac{\partial \boldsymbol{m}}{\partial t}(t_n) + \frac{(\Delta t)^2}{2} \frac{\partial^2 \boldsymbol{m}}{\partial t^2}(t_n) + \mathcal{O}((\Delta t)^3).
\]
For the auxiliary variable $\boldsymbol{g}$, we expand the inverse operator in a Neumann series (valid for small $\Delta t$):
\[
(I - \Delta t \Delta_h)^{-1} = I + \Delta t \Delta_h + (\Delta t)^2 \Delta_h^2 + \mathcal{O}((\Delta t)^3),
\]
thus:
\[
\boldsymbol{g} = \boldsymbol{m} + \Delta t \Delta_h \boldsymbol{m} + (\Delta t)^2 \Delta_h^2 \boldsymbol{m} + \mathcal{O}((\Delta t)^3),
\]
\[
\Delta \boldsymbol{g} = \Delta \boldsymbol{m} + \Delta t \Delta_h^2 \boldsymbol{m} + \mathcal{O}((\Delta t)^2).
\]
Substitute the Taylor expansions into the definition of $\tau^n$:

First, rewrite the time difference term:
\[
\frac{\boldsymbol{m}(t_{n+1}) - \boldsymbol{m}(t_n)}{\Delta t} = \frac{\partial \boldsymbol{m}}{\partial t}(t_n) + \frac{\Delta t}{2} \frac{\partial^2 \boldsymbol{m}}{\partial t^2}(t_n) + \frac{(\Delta t)^2}{6} \frac{\partial^3 \boldsymbol{m}}{\partial t^3}(t_n) + \mathcal{O}((\Delta t)^3).
\]

Second, rewrite the average term:
\[
\frac{\boldsymbol{m}(t_{n+1}) + \boldsymbol{m}(t_n)}{2} = \boldsymbol{m}(t_n) + \frac{\Delta t}{2} \frac{\partial \boldsymbol{m}}{\partial t}(t_n) + \frac{(\Delta t)^2}{4} \frac{\partial^2 \boldsymbol{m}}{\partial t^2}(t_n) + \mathcal{O}((\Delta t)^3).
\]

Third, expand $\Delta_h \boldsymbol{g}(t_{n+1})$ (assuming $\Delta_h$ is consistent with $\Delta$):
\[
\Delta_h \boldsymbol{g}(t_{n+1}) = \Delta_h \boldsymbol{m}(t_{n+1}) + \Delta t \Delta_h^2 \boldsymbol{m}(t_{n+1}) + \mathcal{O}((\Delta t)^2+h^2) = \Delta_h \boldsymbol{m}(t_n)  + \mathcal{O}(\Delta t+h^2).
\]

Substitute these into $\tau^n$:
\begin{align*}
    \tau_h^n &= \left( \frac{\partial \boldsymbol{m}}{\partial t}(t_n) + \frac{\Delta t}{2} \frac{\partial^2 \boldsymbol{m}}{\partial t^2}(t_n) + \mathcal{O}((\Delta t)^2) \right) + \left( \boldsymbol{m}(t_n) + \frac{\Delta t}{2} \frac{\partial \boldsymbol{m}}{\partial t}(t_n) + \mathcal{O}((\Delta t)^2) \right) \times \left( \Delta_h \boldsymbol{m}(t_n) +  \mathcal{O}(\Delta t+h^2) \right)\\
    &=(\partial_t \m+\m\times \Delta_h\m)|_{t=t_n}+O(\Delta t+h^2).
\end{align*}

Now we consider our proposed method, and note that
\begin{align*}
    \frac{\m_h^n-\m_h^n}{\Delta t}=-\frac{\m_h^{n+1}+\m_h^n}{2}\times \Delta_h \g_h^s,
\end{align*}
where $\g_h^{s}=(I-\Delta t \Delta_h)^{-1}\m_h^{s}$, $s=n,n+1$. Such a method with a Gauss-Seidel type iteration to update the $g_h$. From previous local truncation analysis for explicit and implicit schemes, we can directly derive that
\begin{align*}
    \tau_h^n=O(\Delta t+h^2).
\end{align*}

\subsection{Stability of our proposed scheme}
For the step 1, we denote $A$ for its coefficient matrix, and note that the three eigenvalue for $A_i$ and $A_i^{-1}$ are all $1$, so the step 1 is stable.

For step 2, we denote $A$ for its coefficient matrix without any generality. To be specific,
\begin{align}
    A=\begin{pmatrix}
        1 & c_n&-b_{n}\\
        -c_n &1 &a_{n+1}\\
        0 &0&1
    \end{pmatrix}
\end{align}
Note that $\det(A-\lambda I)=(1-\lambda)[(1-\lambda)^2+c_n^2]$, we have the eigenvalues of $A$ below,
\begin{align*}
    \lambda_1=1,\quad \lambda_{2,3}=1\pm ic_n.
\end{align*}
Then $|\lambda_1|=1$, $|\lambda_2|=|\lambda_3|=\sqrt{1+c_n^2}$. Thus, the three eigenvalues of $A_i^{-1}$ are 
\begin{align*}
    \gamma_1=1/\lambda_1,\quad \gamma_2=1/\lambda_2,\quad \gamma_3=1/\lambda_3.
\end{align*}
The absolute values of three eigenvalues of $A_i^{-1}$ are 
\begin{align*}
    |\gamma_1|=1,\quad |\gamma_2|=|\gamma_3|=\frac{1}{\sqrt{1+c_n^2}}\le 1.
\end{align*}
Here $c_n$ is small compared to $1$.

For step 3, we denote $A$ for its coefficient matrix without any generality. To be specific,
\begin{align}
    A=\begin{pmatrix}
        1 & c_n&-b_{n+1}\\
        -c_n &1 &a_{n+1}\\
        b_{n+1} &-a_{n+1} &1
    \end{pmatrix}
\end{align}
Note that
\begin{align*}
    \det(A-\lambda I)=(1-\lambda)^3+(1-\lambda)(a_{n+1}^2+b_{n+1}^2+c_n^2)=(1-\lambda)[(1-\lambda)^2+a_{n+1}^2+b_{n+1}^2+c_n^2].
\end{align*}
Thus the eigenvalues of $A$ are
\begin{align*}
    \lambda_1=1,\quad \lambda_2=1+i\sqrt{a_{n+1}^2+b_{n+1}^2+c_n^2},\quad \lambda_3=1-i\sqrt{a_{n+1}^2+b_{n+1}^2+c_n^2}.
\end{align*}
Then $|\lambda_1|=1$, $|\lambda_2|=|\lambda_3|=\sqrt{1+a_{n+1}^2+b_{n+1}^2+c_n^2}$. Thus, the three eigenvalues of $A_i^{-1}$ are 
\begin{align*}
    \gamma_1=1/\lambda_1,\quad \gamma_2=1/\lambda_2,\quad \gamma_3=1/\lambda_3.
\end{align*}
The absolute values of three eigenvalues of $A_i^{-1}$ are 
\begin{align*}
    |\gamma_1|=1,\quad |\gamma_2|=|\gamma_3|=\frac{1}{\sqrt{1+a_{n+1}^2+b_{n+1}^2+c_n^2}}\le 1.
\end{align*}
Here $a_{n+1},b_{n+1},c_n$ is small compared to $1$.

\begin{remark}
We first define explicit and implicit methods using the general form of time integration for a differential equation:
\[
\frac{du}{dt} = f(t, u), \quad u(t_0) = u_0
\]
where $u(t)$ is the solution variable, $t$ is time, and $f(\cdot)$ is the right-hand side (RHS) function.

Explicit methods compute the next time step $u_{n+1}$ using only \textit{known} values from previous steps ($u_n, u_{n-1}, \dots$):
\[
u_{n+1} = u_n + \Delta t \cdot F_{\text{explicit}}(t_n, u_n)
\]
The most common example is the \textbf{Forward Euler Method}:
\[
u_{n+1} = u_n + \Delta t \cdot f(t_n, u_n)
\]
where $\Delta t = t_{n+1} - t_n$ is the time step size.

Implicit methods compute $u_{n+1}$ using \textit{unknown} values at the next time step ($u_{n+1}$):
\[
u_{n+1} = u_n + \Delta t \cdot F_{\text{implicit}}(t_{n+1}, u_{n+1})
\]
The classic example is the \textbf{Backward Euler Method}:
\[
u_{n+1} = u_n + \Delta t \cdot f(t_{n+1}, u_{n+1})
\]
This requires solving an algebraic equation (linear/nonlinear) for $u_{n+1}$ at each step.

Stability is the most critical advantage of implicit methods. We illustrate this with the linear test problem:
\[
\frac{du}{dt} = \lambda u, \quad \lambda \in \mathbb{C}, \text{Re}(\lambda) < 0
\]

Applying Forward Euler gives:
\[
u_{n+1} = (1 + \lambda \Delta t) u_n
\]
For stability ($|u_{n+1}| \leq |u_n|$), we require:
\[
|1 + \lambda \Delta t| \leq 1
\]
This imposes a \textit{strict upper bound} on $\Delta t$ (e.g., for $\lambda = -10$, $\Delta t < 0.2$). For stiff problems (large $|\lambda|$), $\Delta t$ must be extremely small.

Applying Backward Euler gives:
\[
u_{n+1} = \frac{1}{1 - \lambda \Delta t} u_n
\]
For stability:
\[
\left| \frac{1}{1 - \lambda \Delta t} \right| \leq 1
\]
This holds for \textit{any positive} $\Delta t$ when $\text{Re}(\lambda) < 0$ (A-stable). Implicit methods have \textit{unconditional stability} for stiff problems, while explicit methods have \textit{conditional stability}.
\end{remark}

For our step 3, the matrix 
\begin{align*}
    A=\begin{pmatrix}
        1 & c &-b\\
        -c & 1 & a\\
        b & -a & 1
    \end{pmatrix}
\end{align*}
Here are more comparison details for stability:
\begin{itemize}
    \item For Crank-Nicolson's scheme, $a, b, c$ directly depends on $\m^{n+1}$, so the implicit way gives great stability (A-stability). However, it has to be a nonlinear system.
    \item If $a,b,c$ directly depends on $\m^n$, we get an explicit scheme with CFL-type conditional stability, even though the lower order term is implicit.
    \item If $a,b,c$ indirectly depends on $\m^n$, say implicit treatment by solving a heat diffusion equation, we get a slightly better CFL-condition.
    \item Our method design a Gauss-Seidel type iteration to update the $m^{n+1}$, thus $a, b,c$ can be formed as a mixed way to combine $\m^n$ and $\m^{n+1}$. Such a method gives unconditionally stable results.
\end{itemize}

\subsection{Convergence of our proposed scheme}
By the Lax equivalence theorem, consistency and stability imply convergence. As \(\Delta t, h \to 0\),
\[
\max_{0 \leq n \leq T/\Delta t}\|\boldsymbol{m}^n - \boldsymbol{m}(t^n)\|_{L^2} \to 0,
\]
with first-order accuracy in time and second-order accuracy in space:
\[
\|\boldsymbol{m}^n - \boldsymbol{m}(t^n)\|_{L^2} = O(\Delta t + h^2).
\]

Considering the implicit scheme \cref{eq-im}, we denote $e_h^n=\m_h^n-\m_h(t_n)$, and have

\begin{align*}
    \frac{e_h^{n+1}-\e_h^n}{\Delta t}=-\frac{\e_h^{n+1}+\e_h^n}{2}\times \Delta_h \g_h^{n+1}-\frac{\m(t_{n+1})+\m(t_n)}{2}\times \Delta_h (I-\Delta t\Delta_h)^{-1}\e_h^{n+1}+O(\Delta t+h^2)
\end{align*}
Taking the inner product with $\e_h^{n+1}+\e_h^n$, we have
\begin{align*}
    \|\e_h^{n+1}\|^2-\|\e_h^n\|^2=-\Delta t \langle \frac{\m(t_{n+1})+\m(t_n)}{2}\times \Delta_h (I-\Delta t\Delta_h)^{-1}\e_h^{n+1},\e_h^{n+1}+\e_h^n \rangle +O(\Delta t+h^2).
\end{align*}
Indeed, we have
\[
\|e_h^n\|_{L^2} \leq C (\Delta t + h^2),\quad \forall n\in \mathbb{Z}
\]
which will be proved in the future work.

\begin{lemma}[Boundedness of Elliptic Operator]
For any $\Delta t > 0$ and mesh size $h > 0$, the operator $(I - \Delta t \Delta_h)^{-1}$ is uniformly bounded in $L^2(\Omega)$:
\[
\|(I - \Delta t \Delta_h)^{-1} u\|_{L^2} \leq C \|u\|_{L^2},
\]
where $C$ is a constant independent of $\Delta t$ and $h$. Moreover,
\[
\|\Delta_h (I - \Delta t \Delta_h)^{-1} u\|_{L^2} \leq C \|u\|_{L^2}.
\]
\end{lemma}

\begin{proof}
Since $-\Delta_h$ is positive definite with spectral bound $\sigma(-\Delta_h) \leq C h^{-2}$, we have
\[
\|(I - \Delta t \Delta_h)^{-1}\| = \sup_{\|u\|=1} \frac{1}{\|(I - \Delta t \Delta_h)u\|} \leq \frac{1}{1 + \Delta t \cdot \inf\sigma(-\Delta_h)} \leq C.
\]
The boundedness of $\Delta_h (I - \Delta t \Delta_h)^{-1}$ follows from $\|\Delta_h (I - \Delta t \Delta_h)^{-1}\| = \|-(I - \Delta t \Delta_h)^{-1}(-\Delta_h)\| \leq C$.
\end{proof}



\section{Numerical experiments}
\label{sec:experiments}

In this section, we proceed the accuracy, stability and norm preserving test. We choose different initial conditions to have the robustness of such a method.

\subsection{Accuracy and efficiency tests}

To simplify the accuracy verification, we set the exact solution given for the governing model \cref{eq-5}. Analytical exact solutions are derived for both one-dimensional (1D) and three-dimensional (3D) scenarios to serve as benchmarks for error quantification.

For the 1D case, the exact magnetization solution \(\m_e\) is:
\begin{equation*}
\m_e=\left(\cos(\cos(\pi x))\sin t, \sin(\cos(\pi x))\sin t, \cos t\right)^T,
\end{equation*}
while the corresponding 3D exact solution is:
\begin{equation*}
\m_e=\left(\cos(XYZ)\sin t, \sin(XYZ)\sin t, \cos t\right)^T,
\end{equation*}
where $X=x^2(1-x)^2$, $Y=y^2(1-y)^2$ and $Z=z^2(1-z)^2$.

These exact solutions satisfy the governing equation \cref{eq-5} when the forcing term is defined as \(\f_e=\partial_t \m_e+\m_e \times \Delta \m_e+\alpha\m_e\times (\m_e \times \Delta \m_e)\). They also comply with the homogeneous Neumann boundary condition, ensuring consistency with simulation constraints.

To isolate the temporal approximation error from spatial discretization effects, the spatial resolution in the 1D test is fixed at \(h=5\times 10^{-4}\)—a sufficiently fine grid that makes spatial error negligible compared to temporal error. The Gilbert damping parameter is set to \(\alpha=0.01\), and simulations run until final time \(T=0.1\). Under this configuration, the measured error primarily reflects temporal discretization inaccuracy.

The 3D temporal accuracy test faces inherent constraints from spatial resolution, as excessively fine grids incur prohibitive computational cost. To balance spatial and temporal error contributions, we adopt a coordinated refinement strategy for spatial mesh sizes (\(h_x, h_y, h_z\)) and temporal step-size (\(\Delta t\)) tailored to the proposed method's order: \(\Delta t=h_x^2=h_y^2=h_z^2=h^2=T/N_0\). Here, \(N_0\) is a refinement level parameter, with specific values given in subsequent results. Consistent with the 1D test, \(\alpha=0.01\), and the final time \(T\) is specified later. The first order temporal accuracy is verified from \cref{tab-a-v-3} and \cref{tab-a-10-time-3D-Q-2} for 1D and 3D tests, respectively. 

\begin{table}[htbp]
	\centering
	\caption{The temporal accuracy in 1D test for proposed method with damping $\alpha=0.01$when $h = 5D-4$, $T=1d-1$.}\label{tab-a-v-3}
	\begin{tabular}{|c|c|c|c|}
		\hline
		$k$ & $\|\m_h-\m_e\|_\infty$ & $\|\m_h-\m_e\|_2$ &$\|\m_h-\m_e\|_{H^1}$\\
		\hline
		$T/80$ &0.001304094971804&8.500523347099678e-04&0.006116653503286 \\
		$T/120$ &8.684032607750442e-04&5.745020596842719e-04&0.004123837837786 \\
		$T/160$ &6.505721097687933e-04&4.340035210414707e-04&0.003112167088311 \\
		$T/240$ &4.330118558566187e-04&2.915444972028109e-04&0.002089104867675 \\
		$T/320$ &3.244330910497223e-04&2.195304920518377e-04&0.001572755516106 \\
		\hline 
		order &1.003609279663207&0.976857797102320&0.979916584127108\\
		\hline
	\end{tabular}
\end{table}

\begin{table}[htbp]
	\centering
	\caption{The spatial accuracy in 1D test for proposed method with damping $\alpha=0.01$ when $k= 1D-6$, $T=1d-1$.}\label{tab-a-10-space-Q-1}
	\begin{tabular}{|c|c|c|c|}
		\hline
		$h$ & $\|\m_h-\m_e\|_\infty$ & $\|\m_h-\m_e\|_2$ &$\|\m_h-\m_e\|_{H^1}$ \\
		\hline
		1/16 &4.225596750053739e-04&2.896508432807531e-04&0.002209985483017 \\
		1/24 &1.885253776899853e-04&1.286680596306939e-04&9.768130592826686e-04\\
		1/32 &1.062644247209338e-04&7.252233752489178e-05&5.480246619562391e-04\\
		1/48 &4.739270964135289e-05
		&3.248282203669158e-05&2.426999432605687e-04\\
		1/64 &2.676411577153676e-05&1.848377745692578e-05&1.360493878062051e-04\\
		\hline 
		order &1.990738385102109&1.985405237927322&2.010529053514131 \\
		\hline
	\end{tabular}
\end{table}

Following temporal accuracy evaluation, spatial accuracy tests were conducted to quantify the spatial discretization performance of the proposed method. To prevent temporal errors from interfering with spatial accuracy assessment, the temporal step size was fixed at a sufficiently small \(k=10^{-6}\) for 1D test shown in \cref{tab-a-10-space-Q-1}—making temporal errors negligible compared to spatial errors. We take $k=h^2$, such that the second order spatial accuracy is observed, which is consistent with 1D test.

\begin{table}[htbp]
	\centering
	\caption{The temporal accuracy and spatial accuracy for the proposed method with damping $\alpha=0.01$, $T=0.1$ with $k=h^2$ in 3D. }\label{tab-a-10-time-3D-Q-2}
	\begin{tabular}{|c|c|c|c|c|}
		\hline
		$k$&$h$ & $\|\m_h-\m_e\|_\infty$ & $\|\m_h-\m_e\|_2$ &$\|\m_h-\m_e\|_{H^1}$ \\
	\hline
	T/10&1/10&5.006365255465495e-04&2.886424573026357e-04&3.499752082858884e-04 \\
	T/40 &1/20&1.264524159770852e-04&7.243063478581653e-05&1.290889967619724e-04 \\
	T/57 &1/24&8.912305801656029e-05&5.101958217642231e-05&1.062422830734351e-04\\
	T/78&1/28&6.546837939636063e-05&3.751573538129243e-05&9.169167058822856e-05 \\
	T/102 &1/32&5.037225571857817e-05&2.895218634093047e-05&8.242737703221056e-05 \\
	\hline 
	order &&0.989524179094233&0.991610312809092&0.635976416720683\\
		\hline
	&	order &1.973909640037754&1.978077601727147&1.268736008455239\\
		\hline
	\end{tabular}
\end{table}

\subsection{Norm preserving tests}
We choose the initial condition with below,
\begin{align*}
    \m_0=\left(\cos(\cos(\pi x))\sin (0.01), \sin(\cos(\pi x))\sin (0.01), \cos (0.01)\right)^T,
\end{align*}
and \begin{align*}
\m_0=\left(\cos(XYZ)\sin (0.01), \sin(XYZ)\sin (0.01), \cos (0.01)\right)^T,
\end{align*}
in 1D and 3D, respectively. The results are presented in \cref{tab-8} and \cref{tab-9}.

\begin{table}[htbp]
	\centering
	\caption{The proposed method with damping $\alpha=0.01$ when $h = 5D-4$, $T=1d-1$ for the norm preserving test in 1D.}\label{tab-8}
	\begin{tabular}{|c|c|}
		\hline
		$k$ & $\|\|\m_h\|_2-1\|_\infty$ \\
		\hline
	    2.0D-2 & 1.110223024625157e-15 \\
		1.0D-2 & 2.331468351712829e-15 \\
		5.0D-3 & 2.886579864025407e-15\\
		2.5D-3 & 3.996802888650564e-15 \\
		1.25D-3 & 5.995204332975845e-15 \\
		6.25D-4 & 8.881784197001252e-15 \\
		3.125D-4 & 1.165734175856414e-14\\
		\hline
	\end{tabular}
\end{table}

\begin{table}[htbp]
	\centering
	\caption{The norm preserving test for the proposed method with damping $\alpha=0.01$, $T=0.1$ with $k=h^2$ in 3D. }\label{tab-9}
	\begin{tabular}{|c|c|c|}
		\hline
		$k$&$h$ & $\|\|\m_h\|_2-1\|_\infty$  \\
	\hline
	T/10&1/10& 5.748734821509061e-13\\
	T/40 &1/20&4.660716257376407e-13 \\
	T/57 &1/24&4.194422587033841e-13 \\
	T/78&1/28& 3.410605131648481e-13\\
	\hline 
	\end{tabular}
\end{table}

\subsection{Initial conditions and stability tests}

In this section, we choose different initial conditions specified in each test.

In 1D, we specify the initial condition as
\begin{align*}
    \m_0&=\left(\cos(x^2(1-x)^2)\sin (0.01), \sin(x^2(1-x)^2)\sin (0.01), \cos (0.01)\right)^T\\
    \m_0&=\left(\cos(\cos(\pi x))\sin (0.01), \sin(\cos(\pi x))\sin (0.01), \cos (0.01)\right)^T.
\end{align*}

The results for 1D case using the proposed method to show the numerical profiles are presented in \Cref{fig:1} with the parameters $\alpha=0.01$, $N_x=2000$ and $N_t=5$. 

In 3D, we choose the initial conditions below,
\begin{align*}
\m_0&=\left(\cos(x^2(1-x)^2)\sin (0.01), \sin(x^2(1-x)^2)\sin (0.01), \cos (0.01)\right)^T\\
    \m_0&=\left(\cos(\cos(\pi x))\sin (0.01), \sin(\cos(\pi x))\sin (0.01), \cos (0.01)\right)^T\\
    \m_0&=\left(\cos(\cos(\cos(\pi x)))\sin (\pi x+0), \sin(\cos(\cos(\pi x)))\sin (\pi x+0), \cos (\pi x+0)\right)^T\\
\end{align*}
The results for those initial conditions are presented in \Cref{fig:2-2}, \Cref{fig:2} and \Cref{fig:2-1} which verify the consistency.


\begin{figure}[htbp]
    \centering
    \subfloat[$m_1$]{\includegraphics[width=0.5\linewidth]{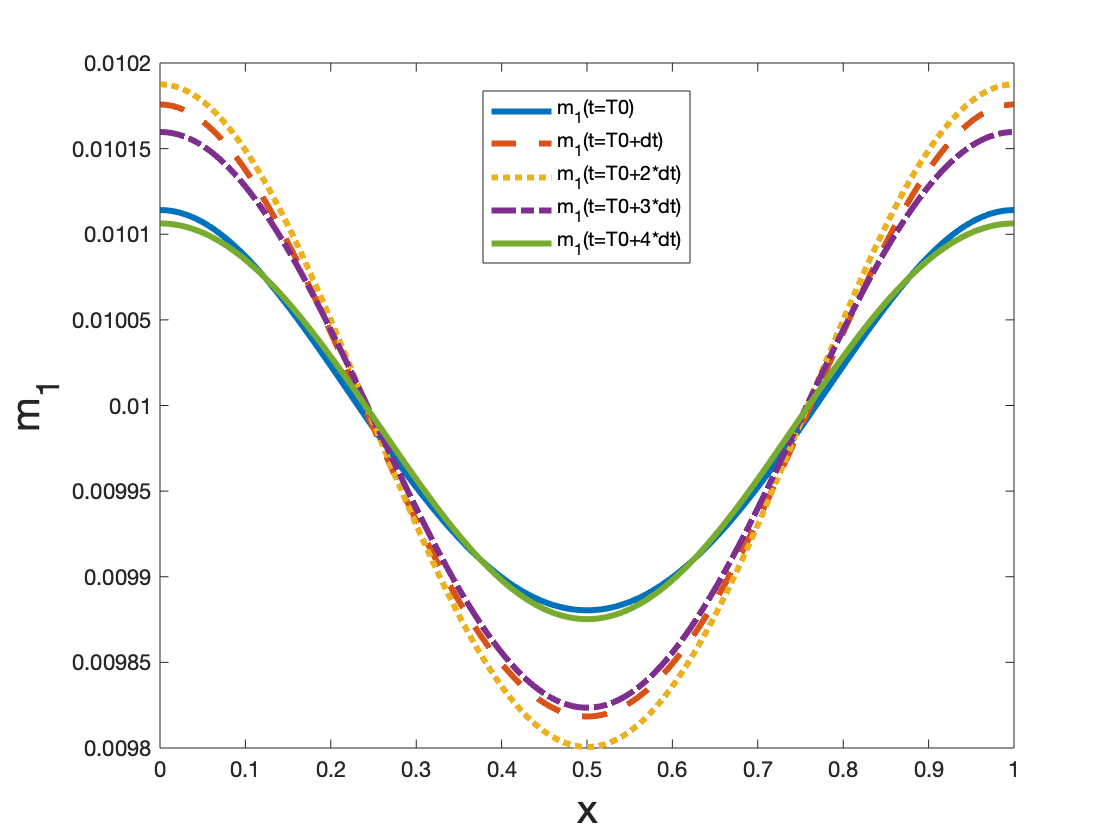}}
    \subfloat[$m_1$]{\includegraphics[width=0.5\linewidth]{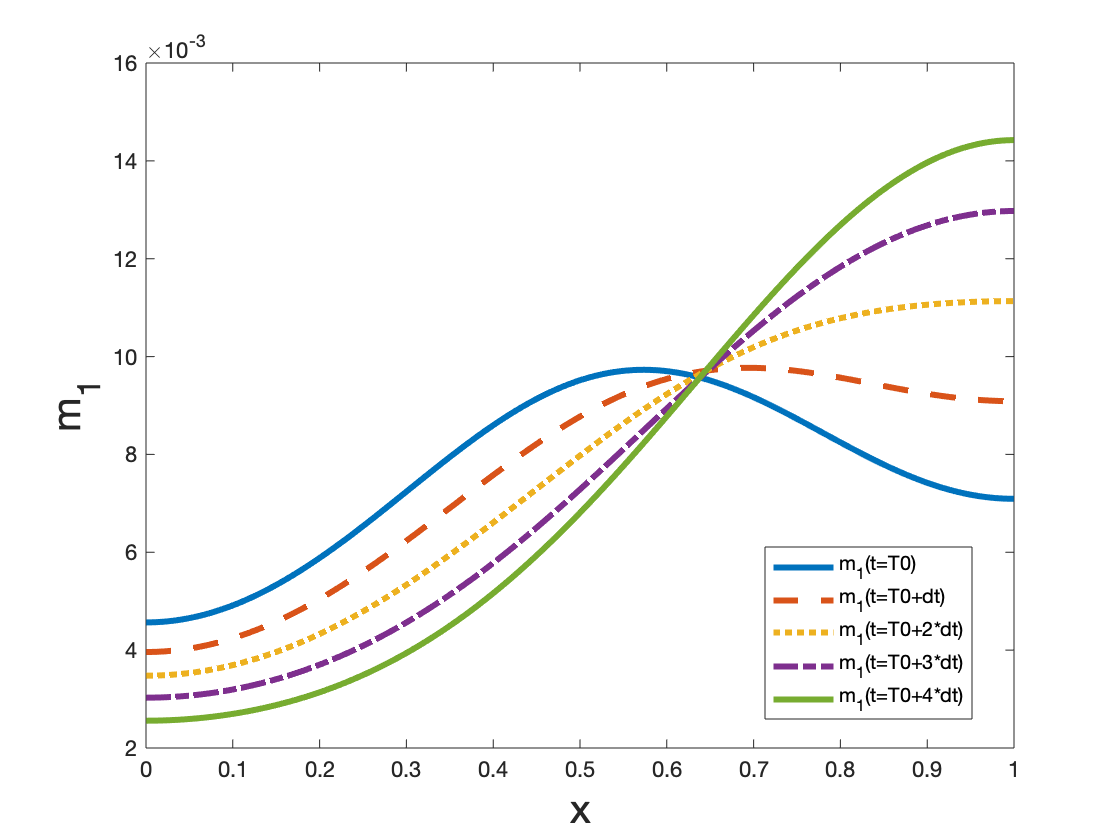}}
    \hspace{0.1in}
    \subfloat[$m_2$]{\includegraphics[width=0.5\linewidth]{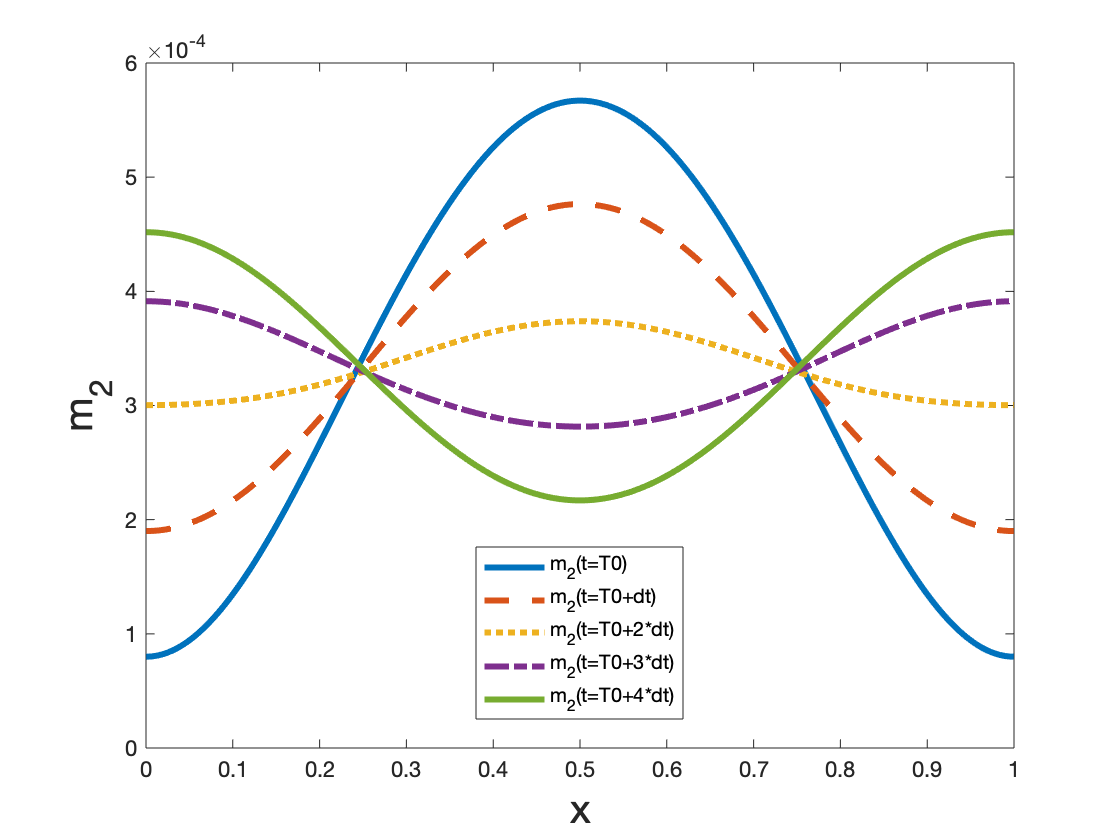}}
    \subfloat[$m_2$]{\includegraphics[width=0.5\linewidth]{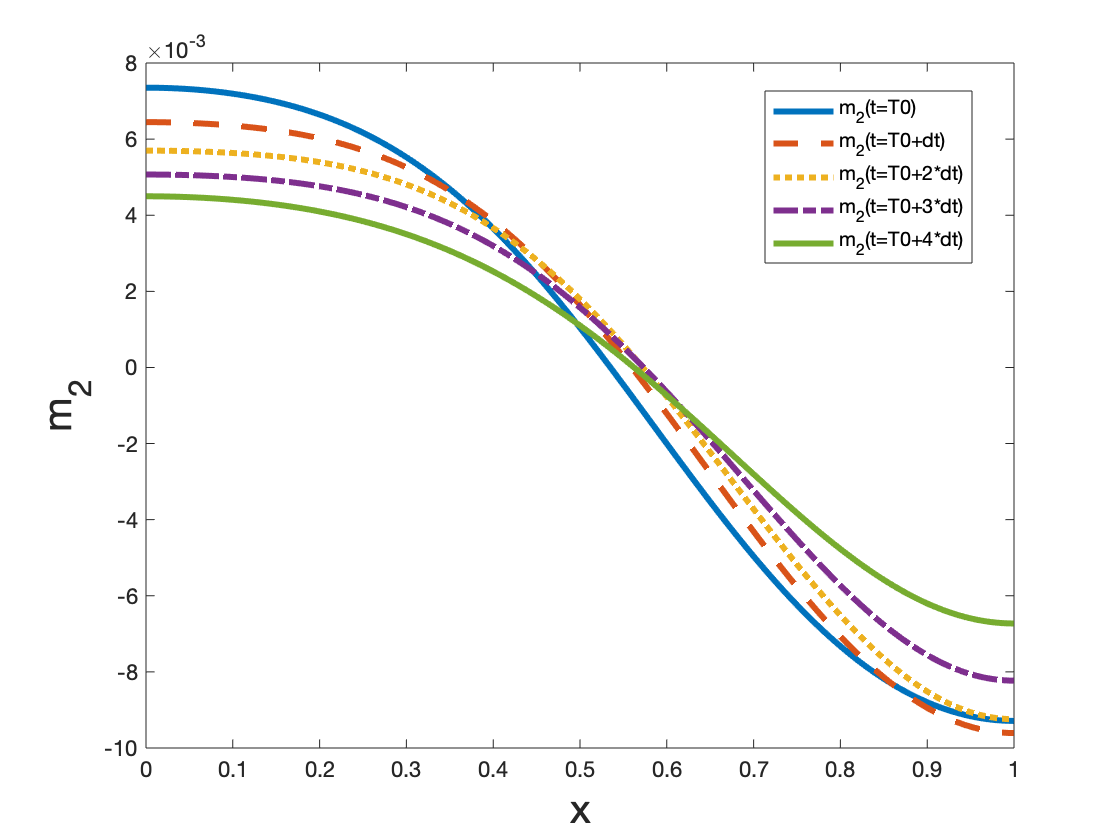}}
    \hspace{0.1in}
    \subfloat[$m_3$]{\includegraphics[width=0.5\linewidth]{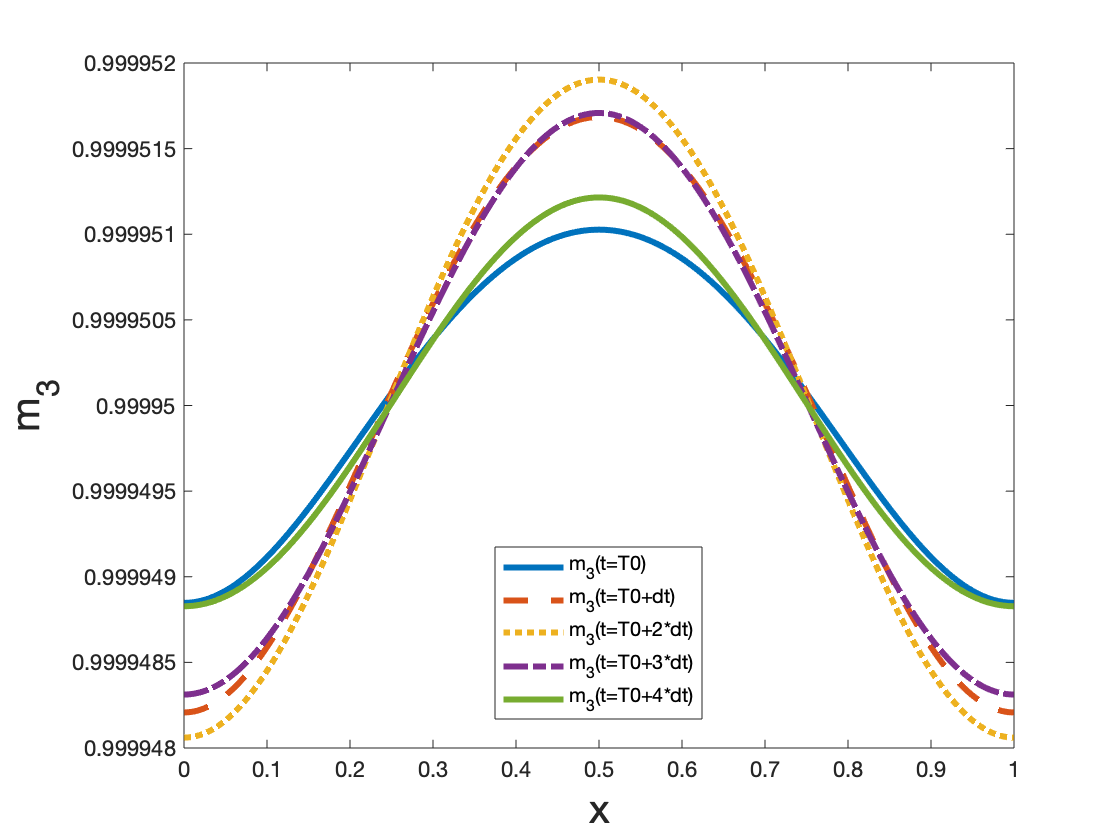}}
    \subfloat[$m_3$]{\includegraphics[width=0.5\linewidth]{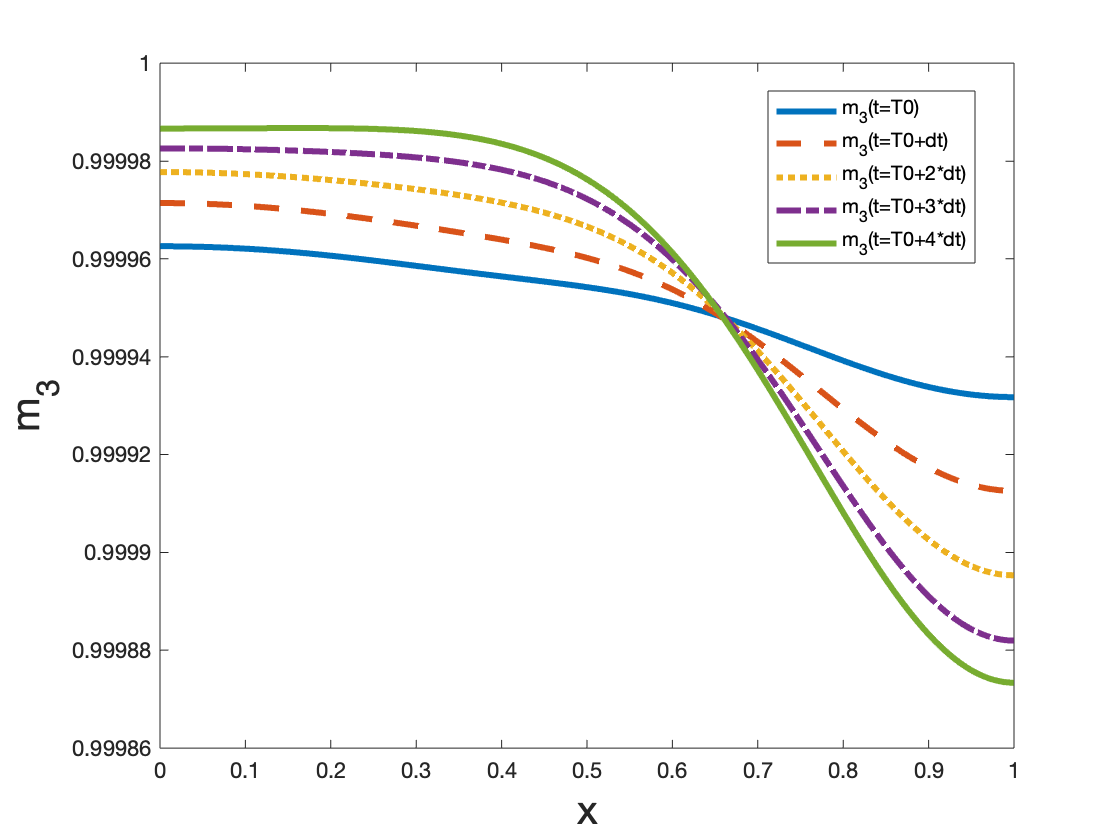}}
    \caption{The solution profile using proposed method in 1D given the initial condition $m_0$ with $T0$ specified without source term, $\alpha=0.01$ and $T=0.1$, $N_x=2000$, $N_t=5$.}
    \label{fig:1}
\end{figure}

\begin{figure}[htbp]
    \centering
    \subfloat[arrow profile]{\includegraphics[width=0.5\linewidth]{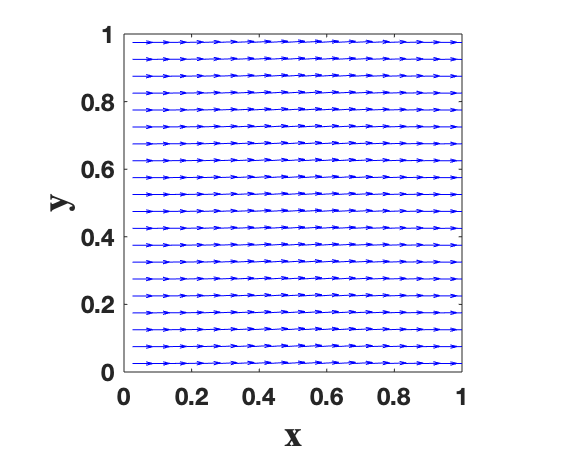}}
    \subfloat[angle profile]{\includegraphics[width=0.53\linewidth]{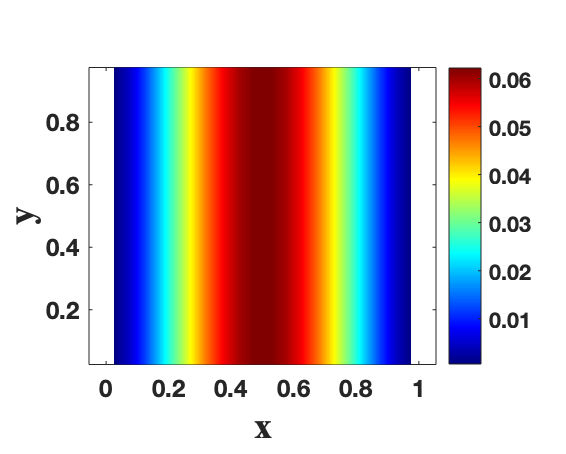}}
    \hspace{0.1in}
    \subfloat[arrow profile]{\includegraphics[width=0.5\linewidth]{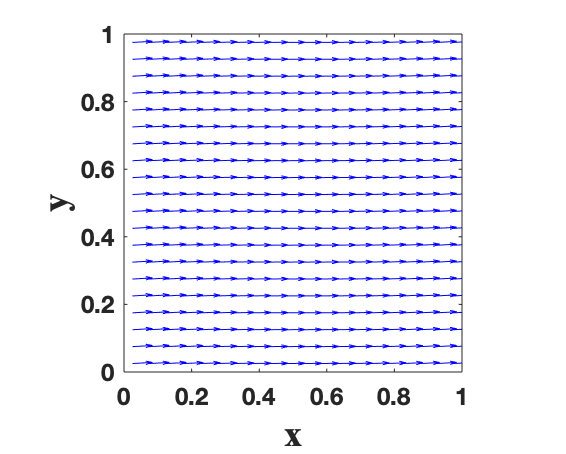}}
    \subfloat[angle profile]{\includegraphics[width=0.53\linewidth]{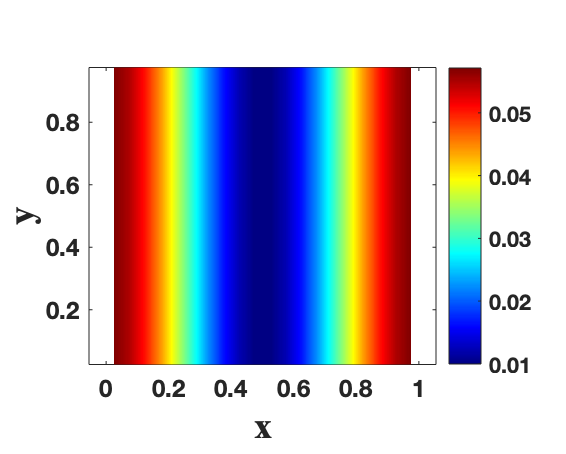}}
      \caption{The solution profile using the proposed methods in 3D given the initial condition $m_0$ with initial condition specified without source term, $\alpha=0$ and $T=0.1$, $N_x=N_y=N_z=20$, $N_t=40$. Top row with initial condition; Bottom row with proposed method. Initial condition given: $\m_0=[\cos(x^2(1-x)^2)\sin(0.01),\sin(x^2(1-x)^2)\sin(0.01),\cos(0.01)]$.}
    \label{fig:2-2}
\end{figure}

\begin{figure}[htbp]
    \centering
    \subfloat[arrow profile]{\includegraphics[width=0.5\linewidth]{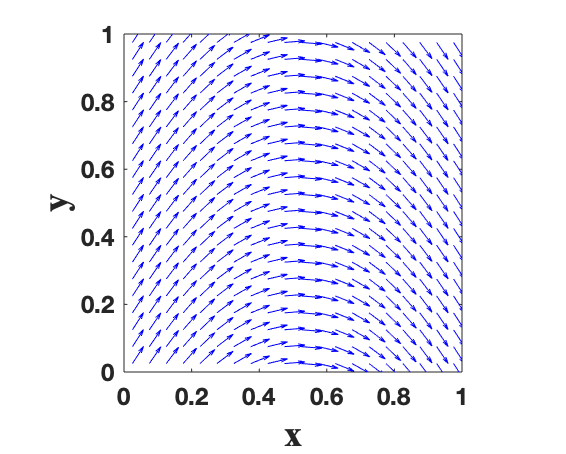}}
    \subfloat[angle profile]{\includegraphics[width=0.53\linewidth]{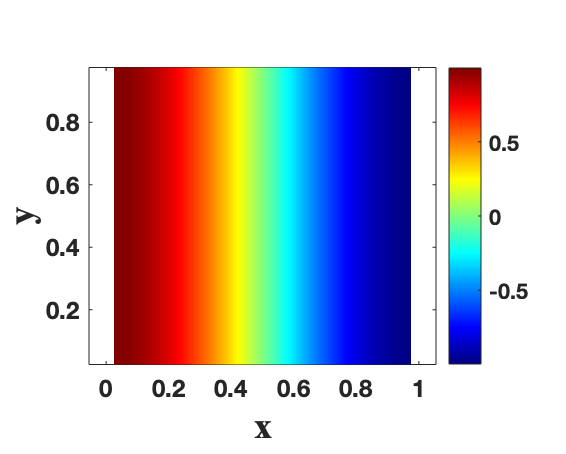}}
    \hspace{0.1in}
    \subfloat[arrow profile]{\includegraphics[width=0.5\linewidth]{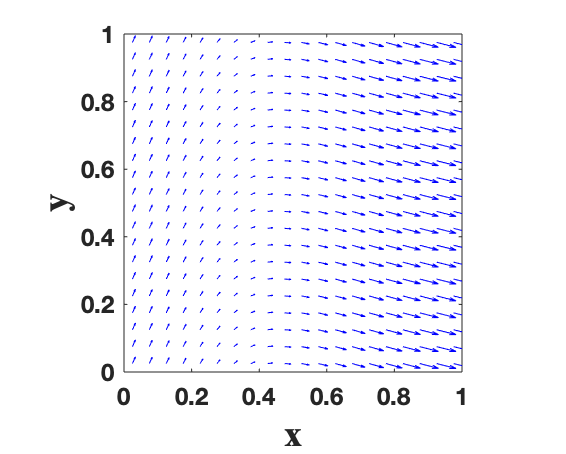}}
    \subfloat[angle profile]{\includegraphics[width=0.53\linewidth]{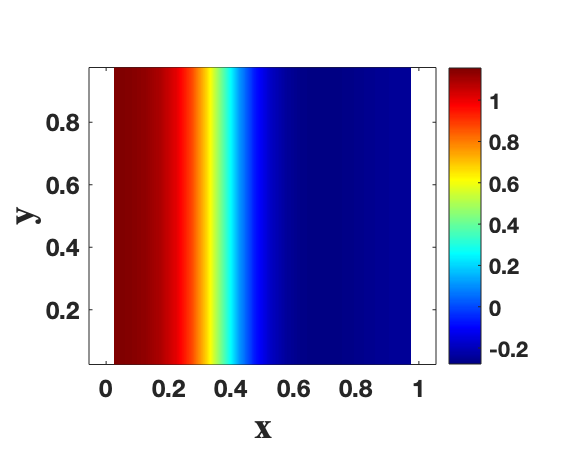}}
      \caption{The solution profile using the proposed methods in 3D given the initial condition $m_0$ with initial condition specified without source term, $\alpha=0$ and $T=0.1$, $N_x=N_y=N_z=20$, $N_t=40$. Top row with initial condition; Bottom row with proposed method. Initial condition given: $\m_0=[\cos(\cos(\pi x))\sin(0.01),\sin(\cos(\pi x))\sin(0.01),\cos(0.01)]$.}
    \label{fig:2}
\end{figure}

\begin{figure}[htbp]
    \centering
    \subfloat[arrow profile]{\includegraphics[width=0.5\linewidth]{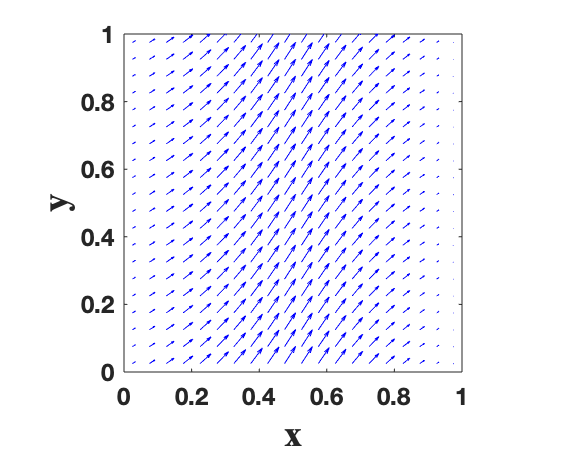}}
    \subfloat[angle profile]{\includegraphics[width=0.53\linewidth]{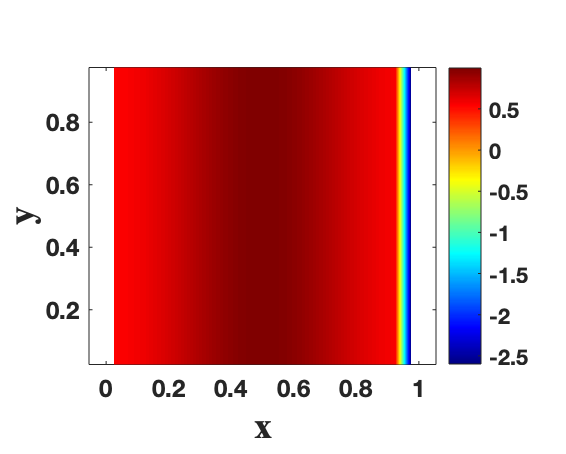}}
    \hspace{0.1in}
    \subfloat[arrow profile]{\includegraphics[width=0.5\linewidth]{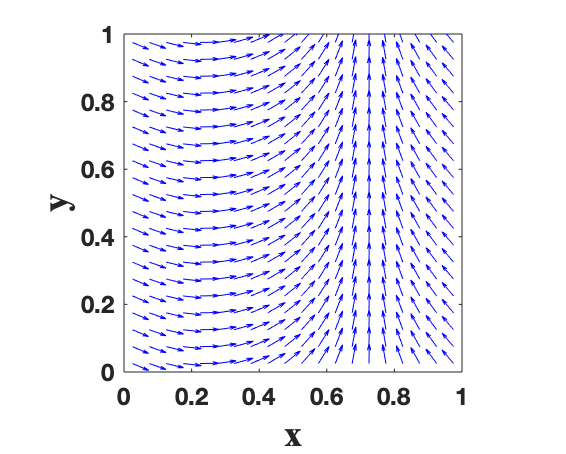}}
    \subfloat[angle profile]{\includegraphics[width=0.53\linewidth]{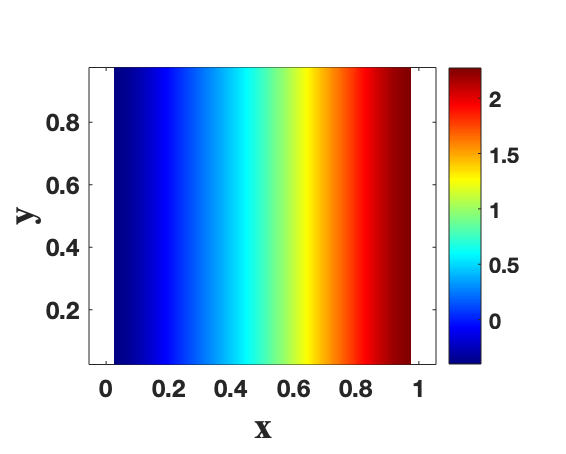}}
      \caption{The solution profile using the proposed methods in 3D given the initial condition $m_0$ with initial condition specified without source term, $\alpha=0$ and $T=0.1$, $N_x=N_y=N_z=20$, $N_t=40$. Top row with initial condition; Bottom row with proposed method. Initial condition given: $\m_0=[\cos(\cos(\cos(\pi x)))\sin(\pi x+T0),\sin(\cos(\cos(\pi x)))\sin(\pi x+T0),\cos(\pi x+T0)]$.}
    \label{fig:2-1}
\end{figure}

\section{Conclusions and discussions}
\label{sec:conclusions}

In this paper, we propose a structure preserving method with the first order accuracy in time and the second order accuracy in space. Such a method use a Crank-Nicolson's implicit method with the Gauss-Seidel iteration. Such a method preserves the norm constraint. Such a method is constructed only based on the equation, not using a projection step. Implicit Crank-Nicolson's method involves a nonlinear system, however, our proposed method is linear and stable. In the future work, we will analyze the stability of such a method and apply for micromagnetics simulations. We will modify such a method to be a finite element methods for the space discretization. The proposed method is stable, length preserving, accurate. Meanwhile, the convergence analysis and stability analysis for the proposed method along with the normalizing step can be proved. The implicit Crank-Nicolson's method has great theoretical value, such as stability, norm preserving, however, our approach is linear and expected to have some of the good properties of the implicit Crank-Nicolson's method.

\section*{Acknowledgments}
This work is supported in part by the Jiangsu Science and Technology Programme-Fundamental Research Plan Fund (BK20250468), Research and
Development Fund of Xi'an Jiaotong Liverpool University (RDF-24-01-015).

\vspace{1cm}

\bibliographystyle{elsarticle-num-names}
\bibliography{references.bib}

@article{cimrak2007survey,
  title={A survey on the numerics and computations for the {Landau-Lifshitz} equation of micromagnetism},
  author={I. Cimr{\'a}k},
  journal={Arch. Comput. Methods Eng.},
  volume={15},
  number={3},
  pages={277--309},
  year={2008},
  publisher={Springer}
}

@article{kruzik2006recent,
  title={Recent developments in the modeling, analysis, and numerics of ferromagnetism},
  author={M. Kruz{\'\i}k and A. Prohl},
  journal={SIAM Rev.},
  volume={48},
  number={3},
  pages={439--483},
  year={2006},
  publisher={SIAM}
}

@article{Landau1935On,
  title={On the theory of the dispersion of magnetic permeability in ferromagnetic bodies},
  author={L.D. Landau and E.M. Lifshits},
  journal={Phys. Z. Sowjet.},
  volume={63},
  number={9},
  pages={153-169},
  year={1935},
}

@article{Gilbert:1955,
  author = {T.L. Gilbert},
  journal = {Phys. Rev.},
  tile ={},
  volume = {100},
  pages = {1243},
  year = {1955},
  note = {[Abstract only; full report, Armor Research Foundation Project No.
    A059, Supplementary Report, May 1, 1956 (unpublished)]}
}

@book{Brown1963micromagnetics,
  title={Micromagnetics},
  author={ W.F. {Brown}},
  year={1963},
  publisher={ Interscience Tracts on Physics and Astronomy. Interscience
Publishers ({John Wiley} and Sons), New York-London},
}

@article{JEONG2010613,
title = {A {Crank–Nicolson} scheme for the {Landau–Lifshitz} equation without damping},
journal = {Journal of Computational and Applied Mathematics},
volume = {234},
number = {2},
pages = {613-623},
year = {2010},
issn = {0377-0427},
doi = {https://doi.org/10.1016/j.cam.2010.01.002},
url = {https://www.sciencedirect.com/science/article/pii/S0377042710000063},
author = {Darae Jeong and Junseok Kim},
}

\end{document}